\documentclass[12pt]{amsart}
\usepackage{amssymb,amsfonts,latexsym,amsmath,amsthm,graphicx}
\usepackage{parskip}
\usepackage{mathrsfs}
\usepackage{mathtools} 
\usepackage[pdftex,bookmarks=true]{hyperref}

\usepackage{mathrsfs}

\usepackage[margin=1in]{geometry}

\numberwithin{equation}{section}
\newtheorem{thm}{Theorem}[section]

\newtheorem{prop}[thm]{Proposition}
\newtheorem{cor}[thm]{Corollary}
\newtheorem{lemma}[thm]{Lemma}
\newtheorem{conj}[thm]{Conjecture}

\newtheorem{remark}{Remark}

\newcommand{\E}{\mathbb{E}}
\renewcommand{\P}{\mathbb{P}}
\newcommand{\RR}{\mathbb{R}}
\newcommand{\R}{\mathbb{R}}
\newcommand{\C}{\mathbb{C}}

\newcommand{\N}{\mathbb{N}}

\newcommand{\sM}{\mathscr{M}}
\newcommand{\sP}{\mathscr{P}}

\usepackage{xcolor}

\usepackage[colorinlistoftodos]{todonotes}

\newcommand{\e}{\varepsilon}

\newcommand{\Var}{\text{Var}}
\newcommand{\I}{\mathcal I}

\newcommand{\be}{\begin{equation}}
\newcommand{\ee}{\end{equation}}
\newcommand{\ep}{\varepsilon}

\begin{document}
 
\title[Bifurcating limit cycles]{The number of limit cycles bifurcating from a randomly perturbed center}
\date{}

\author[M. Krishnapur, E. Lundberg, O. Nguyen]{Manjunath Krishnapur, Erik Lundberg and Oanh Nguyen}
\address{Department of mathematics \\ Indian Institute of Science \\ Bangalore, Karnataka 560012, India}
\email{manju@iisc.ac.in}
\address{Department of Mathematical Sciences \\ Florida Atlantic University \\ Boca Raton, FL 33431, USA}
\email{elundber@fau.edu}
 \address{Division of Applied Mathematics\\ Brown University\\  Providence, RI 02906, USA}
 \email{oanh\_nguyen1@brown.edu}
 \thanks{Lundberg is supported by Simons grant 712397.  Nguyen is supported by NSF grant DMS-2125031}

\begin{abstract} 
We consider the average number of limit cycles that bifurcate from a randomly perturbed linear center where the perturbation consists of random (bivariate) polynomials with independent coefficients.
This problem reduces, by way of classical perturbation theory of the Poincar\'e first return map, to a problem on the real zeros of a random \emph{univariate} polynomial $\displaystyle f_n(x) = \sum_{m=0}^n c_m \xi_m x^m$ with independent coefficients $\xi_m$ having mean zero, variance 1 and $c_m \sim m^{-1/2}$. 
This polynomial belongs to the class of {\it generalized Kac polynomials} at the critical regime.  We provide asymptotics for the average number of real zeros and answer the question on bifurcating limit cycles. Additionally, we provide the correct order of the mean number of real roots in the subcritical regime.
\end{abstract}

\maketitle

\section{Introduction}

\subsection{From ODE to random polynomials}

In the study of ordinary differential equations (ODEs), a \emph{limit cycle} of a system refers to an isolated closed trajectory -- that is, a trajectory that is isolated among the closed trajectories of the system.  Limit cycles are  fundamental in the qualitative theory of nonlinear ODEs, and, in light of the Poincar\'e-Bendixson theorem, they play an especially important role in understanding the dynamics of planar systems. Since ODE systems with polynomial nonlinearities are common in applications, it is desirable to obtain estimates for the number of limit cycles of polynomial systems in terms of the degree of the involved polynomials.  This problem is in general considered to be profoundly difficult \cite{Smale1998}, and most progress has concerned special settings involving perturbative problems, for instance, estimating the number of limit cycles which bifurcate under a polynomial perturbation of a Hamiltonian system. 

A particularly tractable problem, one that can be reduced to enumerating the real zeros of a univariate polynomial, is the case of a perturbed linear center (here we use the notation $\dot{x} = dx/dt$ to denote differentiation of $x=x(t)$ with respect to $t$) 
\begin{equation}\label{eq:pertcenter}
\begin{cases}
\dot{x} = y + \e p(x,y) \\
\dot{y} = -x + \e q(x,y)
\end{cases}
\end{equation}
with $p,q$ being polynomials, and $\e>0$ a small parameter.
For given $p,q$ of degree $d$, the number of bifurcating limit cycles (i.e., those which persist for all sufficiently small $\e>0$) of this system is at most $n = \lfloor (d-1)/2 \rfloor$.  As we review in Section \ref{sec:LC} below, this upper bound can be seen from classical perturbation theory of the Poincar\'e first return map.
We also note that for each degree $d$, there are known examples attaining this upper bound \cite{Teixeira}.

While the extremal behavior for the number of bifurcating limit cycles in the degree-$d$ perturbed center \eqref{eq:pertcenter} is completely understood, it seems natural to inquire about the \emph{typical} number of bifurcating limit cycles. To be precise, we consider the case when the perturbative terms in \eqref{eq:pertcenter}
\begin{equation}
 p(x,y) = \sum_{1 \leq j+k \leq d} \alpha_{j,k} x^j y^k \quad\text{and} \quad q(x,y) = \sum_{1 \leq j+k \leq d} \beta_{j,k} x^j y^k \nonumber
\end{equation}
are random polynomials with $\alpha_{j,k}, \beta_{j,k}$ being independent random variables with mean zero and variance one.  The study of limit cycles in random polynomial systems was initiated by A. Brudnyi in \cite{Brudnyi1}, \cite{Brudnyi2} and the study of {birfucating} ones was investigated by the first author in \cite[Sec. 1.3, Sec. 5]{LC} (see Section \ref{sec:prior}).  However, the above natural choice of randomness was left as an open problem.
 
As we recall in Section \ref{sec:LC}, this problem reduces to studying the number of positive real zeros of the random polynomial
\begin{equation}
f_n(x) = \sum_{m=0}^n c_m \xi_m x^m,\nonumber
\end{equation}
where $n=\lfloor (d-1)/2 \rfloor$, and the coefficients $\xi_m$ are independent with mean zero and variance 1, and the deterministic coefficients $c_m$ are given by (here we use the ``double-factorial'' notation $m!!$ to denote the product of all
positive integers less than or equal to $m$ that have the same parity as $m$)
\begin{equation}\label{eq:sigma_exact}
c_m^2 = \frac{\pi}{2} \sum_{\ell=0}^m
\left(\frac{(2m-2\ell+1)!!(2\ell-1)!!}{(2m+2)!!}\right)^2 + 
 \left(\frac{(2m-2\ell-1)!!(2\ell+1)!!}{(2m+2)!!}\right)^2 ,
\end{equation}
which satisfies (see Section \ref{sec:sigma} below) 
\begin{equation}\label{eq:sigma}
c_m^2 = m^{-1} + O(m^{-2}) \quad \text{as } m \rightarrow \infty.
\end{equation}

The polynomial $f_n$ belongs to the family of {\it generalized Kac polynomials}
\begin{equation}\label{def:f:rho}
f_{n, \rho}(x) = \sum_{m=0}^n c_{m, \rho} \xi_m x^m,
\end{equation}
where $|c_{m, \rho}|$ has order $m^{\rho}$ and $\rho$ is a constant. In other words, the coefficients have a power law growth. We note that \eqref{eq:sigma} corresponds to $$\rho=-1/2,$$
which turns out to be ``critical'' in a sense that we will explain in Section \ref{sec:genKac}.
Here, we note that throughout the paper, the $c_m$ and $c_{m, \rho}$ are allowed to depend on $n$. To give more context, let us provide a (largely incomplete) literature review concerning the number of real roots of this family.

\subsection{Previous results on real roots of generalized Kac polynomials}\label{sec:genKac}

This family has a rich history starting with a series of papers in the early 1900s examining the number of real roots of the Kac polynomials 
$$f_{n, 0}(x) := \sum_{m=0}^n \xi_m x^m.$$ 

In 1932, Bloch and Polya  \cite{BP} considered the Kac polynomial with $\xi_m$ being 
Rademacher, namely $$\P(\xi_m=1)=\P(\xi_m=-1)=1/2$$ and showed that the number of real roots is typically of order $O(\sqrt n)$. 
In the early 1940s, Littlewood and Offord in their ground-breaking series of papers \cite{LO1, LO2, LO3} showed that the number of real roots is in fact poly-logarithmic in $n$.

In 1943, Kac \cite{Kac1943average} discovered his famous formula, nowadays known as the Kac-Rice formula, and for the first time derived the precise asymptotics for the mean number of real roots of the Kac polynomials, with an extra assumption that the random variables $\xi_m$ are iid standard Gaussian. In this case, the mean number of real roots turned out to be 
 $$\left (\frac{2}{\pi} +o(1)\right ) \log n. $$

When the $\xi_m$ are non-Gaussian, the Kac-Rice formula is often hard to handle. 
In fact, the computation of the mean number of real roots for discrete random variables $\xi_m$ required considerable efforts. In 1956, Erd\H{o}s and Offord \cite{EO} found a beautiful yet delicate combinatorial approach to handle the case that the $\xi_m$ are Rademacher, proving that in this case, the number of real roots of the Kac polynomial is $\left (\frac{2}{\pi} +o(1)\right ) \log n$ with high probability. Ibragimov and Maslova in 1960s \cite{IM1, IM2}  successfully extended  the method of Erd\H{o}s-Offord  to treat
the Kac polynomial associated with more general distributions of the $\xi_m$. They showed that if the $\xi_m$ are iid and belong to the domain of attraction of the normal law, then the mean number of real roots of the Kac polynomial $f_{n, 0}$ is also $\left (\frac{2}{\pi} +o(1)\right ) \log n$.

Passing to the generalized Kac family of $(f_{n, \rho})$ which contains all the derivaties of the Kac polynomials and the hyperbolic polynomials, Do, Vu and the second author \cite{DOV} developed the  local universality method, which was initiated in the seminal work of Tao and Vu \cite{TVpoly}, and proved that if the $\xi_m$ are independent with mean 0, variance 1 and finite $(2+\ep_0)$-moments, and $c_{m, \rho} \sim m^{\rho}$ where $\rho>-1/2$ then 
\begin{equation} 
\E N_{f_{n, \rho}}(\R) = (1+o(1)) \frac{1+\sqrt{1+2\rho}}{\pi} \log n ,\nonumber
\end{equation}  
where $N_{f}(I)$ denotes the number of  roots, counted with multiplicity, of a function $f$ in a set $I$. See also \cite{HoiDN, nguyenvuCLT}. The idea of the local universality method is to show that locally, the distribution of the roots of $f_{n, \rho}$ does not depend much on the specific distributions of the $\xi_m$. And so, one can reduce the computation of $\E N_{f_{n, \rho}}$ for general $\xi_m$ to the case that $\xi_m$ are iid standard Gaussian. The latter case can then be handled by the Kac-Rice formula, for instance.

 For $\rho>-1/2$, it is well-known that the roots concentrate near the unit circle and the real roots concentrate near $\pm 1$. So to see why the condition that $\rho>-1/2$ is crucial for this approach, let us look at $f_{n, \rho}(1) = \sum_{m=0}^{n} c_m \xi_m =\sum_{m=0}^{n} \Theta(m^{\rho}) \xi_m$. Its variance is $\sum_{m=0}^{n} \Theta(m^{2\rho}) = n^{\Theta(1)}$ which grows quickly and accounts for the universal behavior of $f_{n, \rho}(1)$, by the classical Central Limit Theorem. For $\rho<-1/2$, this variance is bounded and so for $m$ small, the component $c_m \xi_m$ contributes significantly to $f_{n, \rho}(1)$. Thus, one does not expect that the distribution of $f_{n, \rho}(1)$ is universal, let alone the distribution of its roots. The case that $\rho=-1/2$ is the critical case where the variance of $f_{n, \rho}(1)$ goes to infinity but at a slow rate of $\log n$, as compared to the previous $n^{\Theta(1)}$, which poses significant difficulties. As we have noted above, this is precisely the case encountered in the problem on bifurcating limit cycles, so the first main goal of the current paper will be to adapt the universality method to study the ciritical case $\rho = -1/2$. 
A second main goal of the paper is to provide a panoramic view of the whole generalized Kac model by also addressing the sub-critical case $\rho<-\frac12$ (the sub-critical and super-critical regimes are unrelated to the problem on limit cycles in the way that we have formulated it, but they may arise if the coefficients in the initial set up of the ODE system \eqref{eq:pertcenter} have power law behavior).
 
In parallel, Flasche and Kabluchko \cite{FK, flasche2020real} studied a closely related model of random series with regularly varying coefficients
\begin{equation} 
f_{\infty, \rho}(x) = \sum_{m=0}^{\infty}  c_{m, \rho}\xi_m x^{m}\nonumber
\end{equation}
where the $c_m$ are real deterministic coefficients such that 
\begin{equation}\label{def:varying}
|c_{m, \rho}|= m^{\rho} L(m)
\end{equation}
with $L$ being a slowly varying function. Their result is also restricted to the case that $\rho>-1/2$. Even though their elegant method is different from  the universality method in \cite{DOV}, they also relied on the key property that the function $f_{\infty}(x)$, after being appropriate rescaled, converges to a Gaussian process. This property is a form of universality which is no longer available in the subcritical regime $\rho<-1/2$ and is delicate to achieve at the critical point $\rho = -1/2$.
  
  \subsection{Our results}
As mentioned in the above section, hardly anything is known about the case $\rho\le -1/2$. The only results that we are aware of are by Dembo and Mukherjee \cite{Dembo} and Flasche's thesis \cite{flaschethesis}. In \cite{Dembo}, the authors studied the probability that generalized Kac polynomials have no real roots, which is also known as the persistence probability. In \cite{flaschethesis}, the author studied the mean number of real roots of $f_{\infty, \rho}$ when $\rho=-1/2$. Both papers required the random variables $\xi_m$ to be {\it iid standard Gaussian}.

 In this paper, we consider the general setting where the random variables are not necessarily Gaussian nor necessarily identically distributed. As customary, we assume that the random variables have mean 0, variance 1 and bounded $(2+\ep_0)$-moments. More formally, we define
 
\noindent{\bf Assumption-A}  with parameters $\ep_0>0$ and $C_{2+\ep_0}<\infty$ is said to be satisfied by a real valued random variable $\xi$ if it has zero mean, unit variance and $\E[ |\xi|^{2+\ep_0}]<C_{2+\ep_0}$.

Next, we formulate and generalize equation \eqref{eq:sigma} to general growth power.

 \noindent{\bf Assumption-B} with parameters $\rho\in \R$ and a sequence $a=(a_m)_{m\ge 0}\to 0$   is said to be satisfied by a triangular array $(c_{m,n})_{0\le m\le n}$ of complex numbers if $|m^{-\rho}|c_{m, n}|-1| \le a_m$ for all $m,n$.

 Here is our main result about the mean number of real roots of the generalized Kac polynomial in the critical case $\rho = -1/2$.

\begin{thm}[$\rho =-1/2$]\label{thm:main:critical}
	Let $\ep_0$, $C_{2+\ep_0}$ be positive constants and $(a_m)$ be a sequence converging to $0$. Let
	$$f_n(x) = \sum_{m=0}^n c_{m, n} \xi_m x^m, $$
	be a random polynomial where $\xi_m$ are independent  random variables satisfying Assumption-A with parameters $\ep_0$ and $C_{2+\ep_0}$. Assume that the deterministic coefficients $(c_{m, n})_{0\le m\le n}$ satisfy Assumption-B with parameters $\rho=-\frac12$ and $(a_m)$. 
Then 
	\begin{equation}\label{eq:1toinfinity}
	\E N_{f_n}(1,\infty) = \frac{1+o(1)}{2\pi}\log n,\quad \E N_{f_n}(-\infty, -1) = \frac{1+o(1)}{2\pi}\log n,
	\end{equation}
	\begin{equation}\label{eq:01:uni}
	\E N_{f_n}[-1, 1 ]= o(\log n).
	\end{equation}
Therefore, 
		\begin{equation}\label{eq:0toinfinity}
	\E N_{f_n}(0,\infty) = \frac{1+o(1)}{2\pi}\log n \quad \text{and}\quad \E N_{f_n}(\R) = \frac{1+o(1)}{\pi}\log n.
	\end{equation}
\end{thm}
\begin{remark} The  constants implicit in the $o(1)$ terms in equations \eqref{eq:1toinfinity}, \eqref{eq:01:uni} and \eqref{eq:0toinfinity} depend only on  the parameters in Assumptions A and B (i.e., on $\ep_0$, $C_{2+\ep_0}$, $(a_m)$). For the rest of the paper, we refer to $(a_m)$ as ``the rate of convergence of $m^{-\rho}|c_{m, n}|$".
\end{remark}

Applying {Theorem \ref{thm:main:critical}} to the setting of limit cycles, we obtain the following answer to the problem posed at the beginning of the paper.
\begin{cor}\label{cor:limit:cycle}
	Suppose $\alpha_{j,k}, \beta_{j,k}$ are independent random variables satisfying Assumption-A with parameters $\ep_0>0$ and $C_{2+\ep_0}<\infty$. 
	Then the average number of bifurcating limit cycles of the system \eqref{eq:pertcenter}
	is asymptotically $\frac{1}{2\pi} \log d$.
	Moreover, the average number of bifurcating limit cycles residing inside the unit disk is $o(\log d).$
\end{cor}

We can also apply Theorem \ref{thm:main:critical} to so-called Li\'enard ODE systems 
\begin{equation}\label{eq:Lienard}
\begin{cases}
	\dot{x} = y - \e p(x) \\
	\dot{y} = -x
\end{cases},
\end{equation}
where $p(x)$ is a random polynomial
\begin{equation}\label{eq:p}
p(x) =  \sum_{k=1}^d \alpha_k x^k,
\end{equation}
with independent coefficients, and we consider  again  the \emph{bifurcating} limit cycles, i.e., those which persist for arbitrarily small $\e>0$. Interestingly, as stated in the following corollary, the perturbed system \eqref{eq:Lienard} results in the same asymptotic as \eqref{eq:pertcenter} even though the perturbation in \eqref{eq:Lienard} depends on a single variable and only affects a single component of the ODE system.

\begin{cor}\label{cor:Lienard}
	Suppose the coefficients $\alpha_k$ in \eqref{eq:p} are  independent random variables satisfying Assumption-A with parameters $\ep_0>0$ and $C_{2+\ep_0}<\infty$. 
	The average number of bifurcating limit cycles of the system \eqref{eq:pertcenter}
	is asymptotically $\frac{1}{2\pi} \log d$ as $d \rightarrow \infty$.
\end{cor}

One natural question that remains unanswered in Theorem~\ref{thm:main:critical} is the precise order of growth for the expected number of roots in $[ -1,1]$. For the case of Gaussian coefficients,  this quantity is $O(\sqrt{\log n})$ (Theorem~\ref{thm:gaussian} ). We conjecture that the same should hold for general random variables at  $\rho=-\frac12$ (Conjecture~\ref{conj:critical}). Another natural question is what happens for $\rho<-\frac12$. In Theorem~\ref{thm:main:sub} we show that for general random variables, the expected number is $O(1)$. Together with the results from \cite{DOV} covering the super-critical case, Theorems \ref{thm:main:critical} and \ref{thm:main:sub} cover the full range of power-law rates $\rho \in \RR$.  The following table presents a summary of what we now know for general random variables.
\begin{table}[h!]
	\begin{center}
		\label{tab:table1}
		\begin{tabular}{c|c|c} 
			$\rho$ & $\E N(0, 1)$ & $\E N(1, \infty)$ \\
			\hline
			 $>-1/2$& $\frac{\sqrt{2\rho+1}+o(1)}{2\pi}\log n$  & $\frac{1+o(1)}{2\pi}\log n$\\
			$=-1/2$ & $o(\log n)^{1-\ep}$ & $\frac{1+o(1)}{2\pi}\log n$\\
			$<-1/2$ & $\Theta(1)$ & $\frac{1+o(1)}{2\pi}\log n$\\ \hline
		\end{tabular}
	\end{center}
		\caption{Expected number of real zeros of random polynomials with general coefficient distributions.}
\end{table}

The precise growth rate for the Gaussian setting is as follows.
 \begin{thm}[Gaussian] \label{thm:gaussian} Let $(a_m)$ be a sequence converging to 0 and $\rho\in \R$. Assume that $(c_{m,n})_{0\le m\le n}$ satisfy Assumption-B with parameters $\rho$ and $(a_m)$. Assume that $\xi_m$ are i.i.d. standard Gaussians. Then for \begin{equation} 
	f_{n, \rho}(x) = \sum_{m=0}^n c_{m, \rho} \xi_m x^m,\nonumber
	\end{equation}
	we have
	\begin{equation}\label{eq:gau}
	\E N_{f_{n, \rho}}(0, 1) = \begin{cases}
	\frac{\sqrt{2\rho+1}+o(1)}{2\pi}\log n \quad\text{if $\rho>-1/2$}\\
	\frac{1+o(1)}{\pi} \sqrt{\log n} \quad\text{if $\rho=-1/2$}\\
	\Theta(1) \quad\text{if $\rho<-1/2$}.
	\end{cases}
	\end{equation}
\end{thm}
Let $E_{\rho} = \E N_{f_{n, \rho}}(0, 1)$ in this case. We visualize this result in Figure \ref{figE}.

\begin{remark} \label{rmk:root}  
 	\begin{figure}[h]
		\centering
		\includegraphics[width=0.8\linewidth]{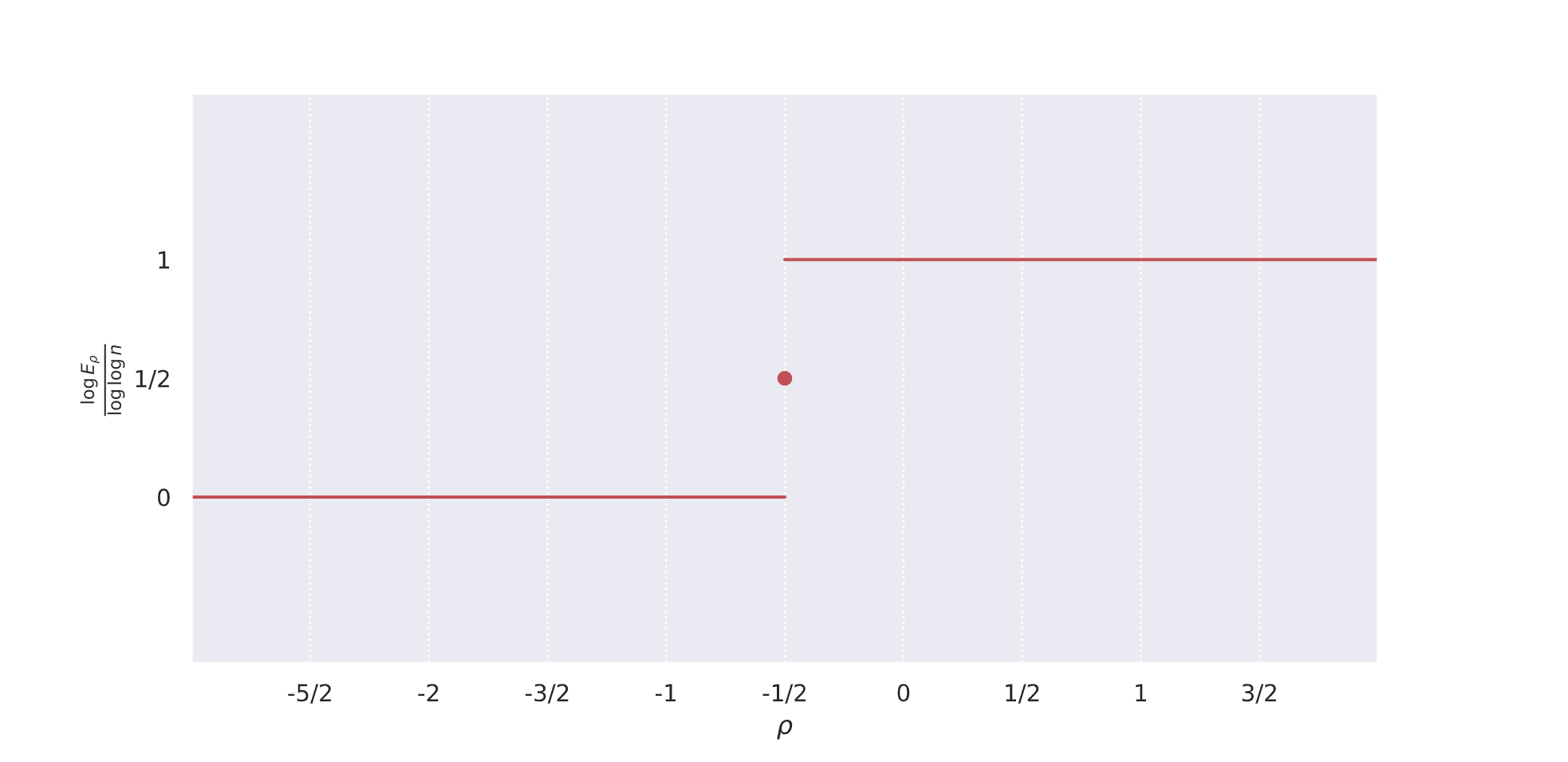}
		\caption{The growth of $E_{\rho}$ as a power of $\log n$.} \label{figE}
	\end{figure}
We find the abrupt change in the growth rate of the expected number of zeros in $(0,1)$ around the critical value $\rho=-1/2$ to be quite surprising (see Figure \ref{figE}). 
	Furthermore, noting that the derivative of $f_{n, \rho}$ is indeed $f_{n, \rho+1}$, this equation implies that for $-3/2\le \rho \le -1/2$, the random polynomial $f_{n, \rho}$ surprisingly has way more critical points (which are the roots of $f_{n, \rho+1}$) than it has zeros in $(0, 1)$.  
\end{remark}

In the subcritical regime $\rho<-1/2$, we show that the mean number of roots in $[-1, 1]$ is $O(1)$, which agrees with the Gaussian case \eqref{eq:gau}. We in fact prove a stronger statement that not only the first moment is $O(1)$ but also all moments.
\begin{thm} [$\rho<-1/2$]\label{thm:main:sub}
Let $\ep_0>0$, $C_{2+\ep_0}>0$, $\rho<-\frac12$ be constants, and $(a_m)$ be a sequence converging to $0$. Let
	$$f_n(x) = \sum_{m=0}^n c_{m, n} \xi_m x^m, $$
	be a random polynomial where $\xi_m$ are independent  random variables satisfying Assumption-A with parameters $\ep_0$ and $C_{2+\ep_0}$. Assume that the deterministic coefficients $(c_{m, n})_{0\le m\le n}$ satisfy Assumption-B with parameters $\rho$ and $(a_m)$. 
Then,   for all $\ell>0$,
	 \begin{equation}\label{eq:sub:01}
	 \E N^{\ell}_{f_n}[-1, 1] = O_{\ell}(1).
	 \end{equation}
	 For the outside intervals, we have
	 \begin{equation}\label{eq:sub:1inf}
	 \E N_{f_n}(1,\infty) = \frac{1+o(1)}{2\pi}\log n,\quad \E N_{f_n}(-\infty, -1) = \frac{1+o(1)}{2\pi}\log n,
	 \end{equation}
	 and
	 \begin{equation}\label{eq:sub:R}
	 \E N_{f_n}(0,\infty) = \frac{1+o(1)}{2\pi}\log n \quad \text{and}\quad \E N_{f_n}(\R) = \frac{1+o(1)}{\pi}\log n.
	 \end{equation}
 \end{thm}

 Finally, we conjecture that what happens for the Gaussian case \eqref{eq:gau} for $\rho=-1/2$ also holds for general distributions.
 \begin{conj}\label{conj:critical}
 	 Under the assumption of Theorem \ref{thm:main:critical}, we have
 	 \begin{equation}\label{eq:01:uni:conj}
 	 \E N_{f_n}(-1, 1) = \frac{1+o(1)}{\pi} \sqrt{\log n} .\nonumber
 	 \end{equation}
 \end{conj}

\subsection{Relevant prior work on limit cycles}\label{sec:prior}

The study of limit cycles in random polynomial systems was initiated by A. Brudnyi in \cite{Brudnyi1}, \cite{Brudnyi2} where the author considers the randomly perturbed center with the degree $d$ of the perturbation allowed to increase as the size $\e$ of the perturbation shrinks.  Sampling the vector of coefficients of $p,q$ uniformly from the unit ball and taking $\e = \e(d) = O(d^{-1/2}) $ as $d \rightarrow \infty$, it is shown that the expected number of limit cycles residing in the disk of radius $1/2$ is $O(\log d)$.
In \cite{LC}, the first author established a probabilistic limit law in a related setting, showing that when $p,q$ have iid coefficients and $\e=\e(d) = o(1)$ as $d \rightarrow \infty$, the number of limit cycles in a disk of radius $\rho<1$ converges almost surely to the number of zeros in the interval $(0,\rho)$ of a certain univariate random power series.

Concerning the problem at hand, the study of \emph{birfucating} limit cycles (i.e., letting $\e \rightarrow 0$ before taking $d \rightarrow \infty$) for random polynomial systems was initiated only recently by the first author \cite{LC}. When $p,q$ are sampled from the so-called Kostlan ensemble, it was shown in \cite{LC} that the average number of bifurcating limit cycles within the disk of radius $\rho$ is $\frac{\arctan{\rho}}{\pi} d(1+o(1))$ as $d \rightarrow \infty$.  It was also shown that when $p,q$ have independent coefficients with variance having power-law growth rate $m^\gamma$ with $\gamma>0$ and $m$ the degree of the associated monomial, the average number of bifurcating limit cycles is $ \left( \frac{1+\sqrt{\gamma}}{2\pi} \log d \right) (1+o(1))$ as $d \rightarrow \infty$.  The assumption $\gamma>0$ was important for the method of proof used in \cite{LC} which relied on the aforementioned results from \cite{DOV}.  This is why $\gamma=0$, which includes the important case when $p,q$ have iid coefficients, was left as an open problem (now addressed by our results stated above).
 
 \subsection{Proof ideas}
As mentioned above, when $\rho=-1/2$, the variance of $f_n(1)$ is of order $\log n$ which still converges to $\infty$ though with a slow rate. Therefore, it is still reasonable to expect that $f_n$ continues to have the universality property around 1. So, our goal is to establish that in order to show Theorem \ref{thm:main:critical}. One major difficulty for the analysis is to control the supremum of $f_n$ over small balls $B(x_0, r)$ inside (but near {the boundary of}) the unit disk, and bound it by, roughly speaking, some power of $\sqrt{\Var f_n(x_0)}\approx \log \frac{1}{1-x_0}$. One of the common techniques (used in \cite{DOV, nguyenvurandomfunction17}, for example) is based on the simple bound
   $$\sup_{x\in B(x_0, r)} |f_n(x)|\le \sum_{m=0}^{n} |c_{m, n}||\xi_m|(x_0+r)^{m}\approx \sum_{m=0}^{n} m^{-1/2}(x_0+r)^{m}\approx \frac{1}{\sqrt{1-x_0}}.$$ 
   However, the RHS is too large compared to any power of $\sqrt{\Var f_n}$ for it to be useful. This issue stems again from the fact that $\Var f_n(x)$ {grows very slowly}. Another common approach is to use an $\ep$-net argument, and this often involves taking the derivative of $f_n$. As we have discussed in Remark \ref{rmk:root}, for $\rho=-1/2$, $f_n'$ corresponds to $f_{n, \rho=1/2}$ which is already in the supercritical regime that has many more real roots than $f_n$ and is typically  much larger than $f_n$. This limits the usefulness of estimates that are based on passing to the derivative.
   In this paper, we come up with a simple solution that more or less combines both approaches together with a key estimate using Taylor expansion to an appropriately chosen degree $m$ (see Section \ref{sec:cond:bddn:0}).
 
 For the roots in $(1, \infty)$, the behavior of $f_n(x)$ where $x$ is very near 1, say $x\in (1, 1+\log n/n)$, may be quite erratic. To overcome this issue, we show that the contribution from this interval very close to 1 is negligible. And then outside of this interval, we show that $f_n$ behaves just like those in the supercritical regime. We in fact show that the same phenomenon holds even for $\rho<-1/2 $ (see Lemma \ref{lm:allrho:1}).

  For roots in $(0, 1)$ and for $\rho<-1/2$, we adapt a classical approach to show that with very high probability, $f_n$ does not have many roots in a small interval.  
 \subsection{Notations.} Throughout the paper, $N_{f}(S)$ denotes the number of roots of a function $f$ in a set $S$, counted with multiplicities. 
 We use standard asymptotic notations under the assumption that $n$ tends to infinity. For two positive  sequences $(a_n)$ and $(b_n)$, we say that $a_n \gg b_n$ or $b_n \ll a_n$ if there exists a constant $C$ such that $b_n\le C a_n$. Equivalently, we also write $b_n=O(a_n)$ and $a_n=\Omega(b_n)$. If $a_n\ll b_n$ and $b_n\ll a_n$, we write $a_n = \Theta(b_n)$.
 If $|c_n|\ll a_n$, we also write $c_n\ll a_n$.  
 
 
 The rest of the paper is organized as follows. In Section \ref{sec:pf:cor}, we prove Corollaries \ref{cor:limit:cycle} and \ref{cor:Lienard} after reviewing how bifurcating limit cycles of \eqref{eq:pertcenter} relate to the study of zeros of a univariate polynomial. In Section \ref{sec:pf:main}, we prove Theorem \ref{thm:main:critical} (the critical case $\rho = -1/2$). Theorem~\ref{thm:gaussian} (the Gaussian case) is  proved in
  Section \ref{sec:gau}. We prove Theorem \ref{thm:main:sub} (the sub-critical case $\rho < -1/2$) in Section \ref{sec:main:sub}.
 
\section{Proofs of the Corollaries} \label{sec:pf:cor}

\subsection{Bifurcating limit cycles}\label{sec:LC}

As indicated in the introduction, the possibility to apply Theorem \ref{thm:main:critical} to the problem of counting limit cycle bifurcations of the perturbed center \eqref{eq:pertcenter} rests on the (deterministic) fact that the latter problem reduces, in the generic case, to counting positive real zeros of an associated univariate polynomial. This reduction can be seen from perturbation theory of the Poincar\'e first return map as follows. First note that the trajectory of \eqref{eq:pertcenter} from some initial condition on the positive $x$-axis will, for $\e>0$ suffiiciently small, wind around the origin in the counterclockwise direction returning to the positive $x$-axis.  Let $\sP$ denote the Poincar\'e map associated to \eqref{eq:pertcenter}, sending an initial condition on the positive $x$-axis to the position of its first return to the positive $x$-axis.
That is, we have $r \rightarrow \sP(r)$ if the trajectory with initial condition $(r,0)$ first returns to the positive $x$-axis at position $(\sP(r),0)$.

The Poincar\'e map $\sP$ admits a perturbation expansion in powers of $\e$
\begin{equation}\label{eq:pertexp}
\sP(r) = r + \e \sM(r) + O(\e^2),
\end{equation}
where the function $\sM(r)$ appearing in the first-order correction is the Poincar\'e-Pontryagin-Melnikov integral
\begin{equation}\label{eq:Melnikov}
\sM(r) = \int_{C_r} p dy - q dx , \quad C_r :=\{ x^2+y^2 = r^2 \}.
\end{equation}
For a rigorous derivation of the perturbation expansion \eqref{eq:pertcenter} within a more general context, we refer the reader to \cite[Sec.  26]{IlyashenkoBook}, but let us briefly provide some intuition here.
The zeroth order term $r$ in \eqref{eq:pertexp} simply follows from setting $\e=0$ in the system \eqref{eq:pertcenter} which results in a family of circular trajectories centered at the origin, and this clearly corresponds to an identity map $r \rightarrow r$ for the Poincar\'e first return map. The Melnikov integral $\sM(r)$ appearing in the first-order correction is nothing other than the flux integral $\int_{C_r} \binom{p}{q} \cdot \hat{n} \, ds$ of the vector field $\binom{p}{q}$ across the zeroth-order circular trajectory of radius $r$ (here $\hat{n}$ denotes the outward-pointing unit normal vector and $ds$ denotes the arclength element).  For $\e>0$ small, the flux $\e \sM(r)$ accounts for the net displacement, to first order in $\e$, caused by the perturbative component $\e \binom{p}{q}$ directing the true trajectory away from the zeroth-order approximation.  Intuitively, this expresses net displacement in the normal direction as a superposition of disturbances which might seem surprising in the context of a nonlinear system, but when $\e$ is small the trajectory remains uniformly close to the zeroth order circular trajectory (at least up until the time of its first return to the positive $x$-axis) so that the nonlinearity may be treated (to first order in $\e$) as one would treat a nonhomogeneity---by taking a superposition of disturbances.

\begin{figure}[h]
\centering
\includegraphics[width=0.4\linewidth]{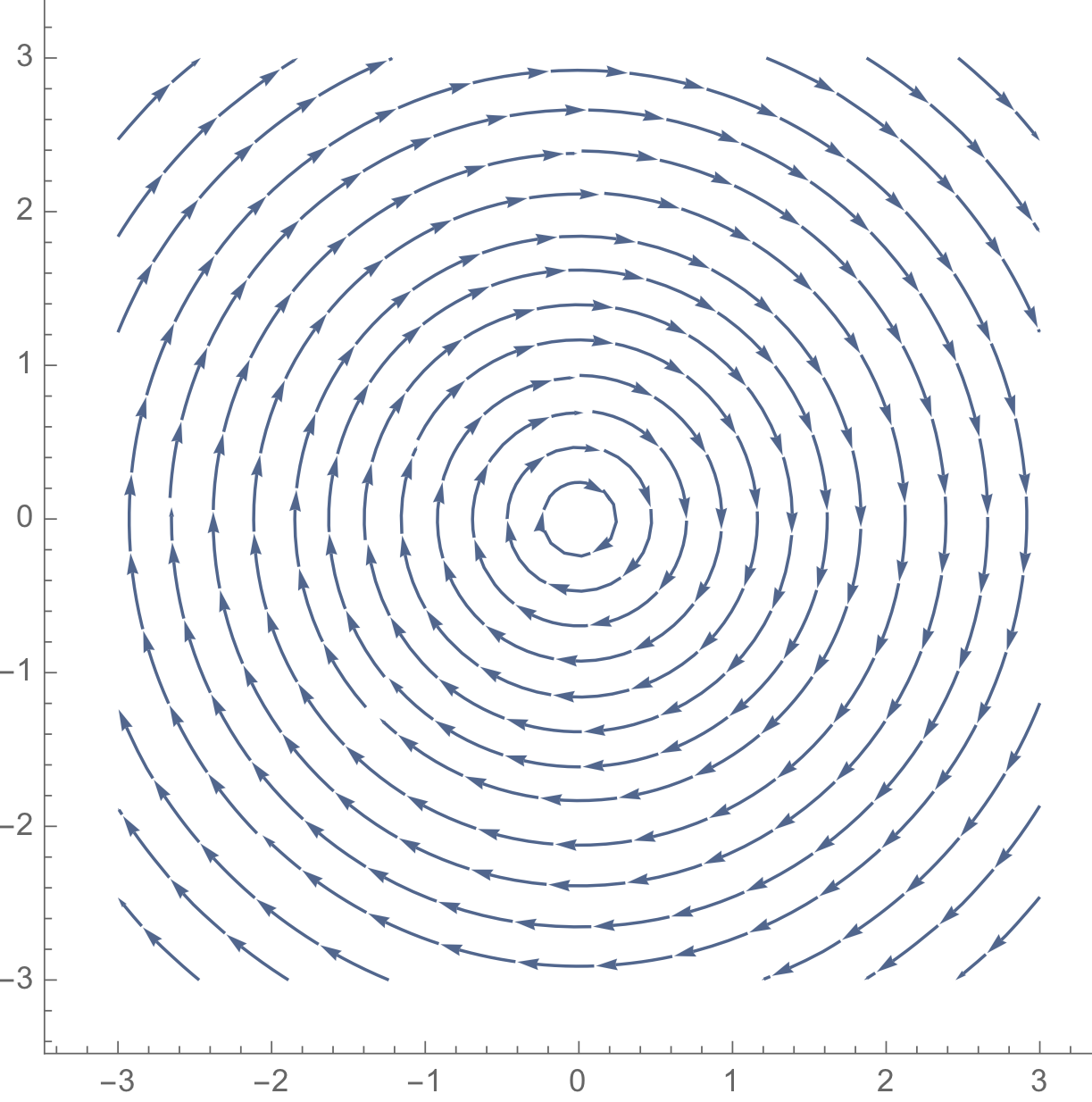} \hspace{0.3in}
\includegraphics[width=0.4\linewidth]{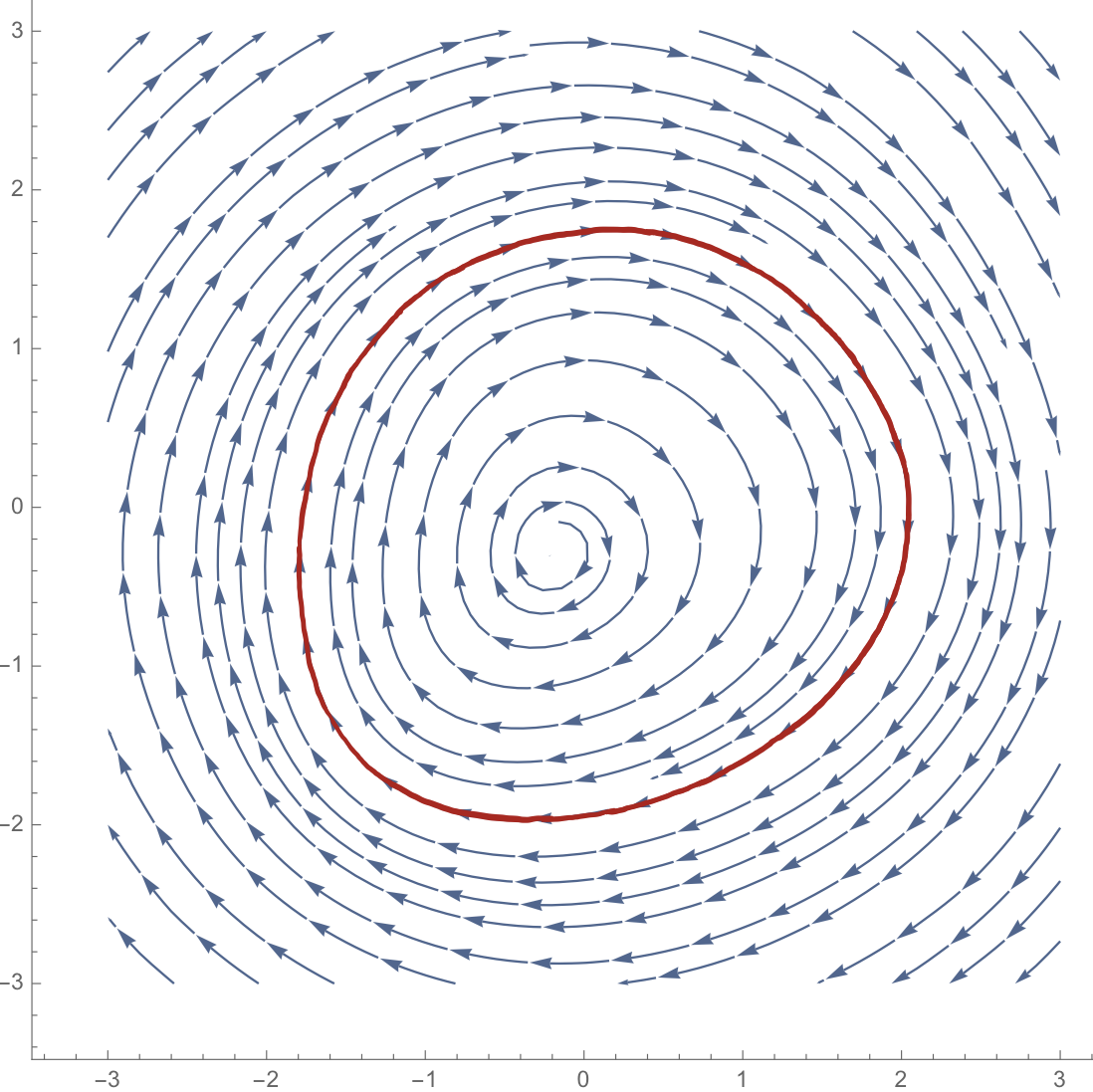}
\caption{A numerical plot (using Mathematica's StreamPlot function) showing some trajectories for the linear center (left) and for the linear center perturbed by randomly sampled polynomials of degree $d=10$ with a limit cycle shown in bold (right).}
\label{fig:perturb}
\end{figure}

Notice from the definition of the Poincar\'e map that a trajectory of the ODE system \eqref{eq:pertcenter} passing through an initial point $(r_0,0)$ is a closed trajectory if and only if $r_0$ is a fixed point of the Poincar\'e map, i.e., we have $\sP(r_0) = r_0$ or equivalently $r_0$ is a zero of $\sP(r)-r$.  Moreover, the limit cycles of the system correspond to the isolated fixed points of the Poincar\'e map $\sP$.  From \eqref{eq:pertexp} we have 
$$
\frac{\sP(r) - r}{\e} = \sM(r) + O(\e).
$$
If $\sM$ has nondegenerate zeros, and if the error term $O(\e)$ is not just uniformly small but also obeys some $C^1$-smallness as a function of $r$, then the transversality and stability principle will guarantee that the fixed points of $\sP$ are all isolated and are in one-to-one correspondence with the positive real zeros of $\sM$.
In fact, as stated in the following result this conclusion holds merely assuming that $\sM$ does not vanish identically (this is stronger than we need as the zeros will be nondegenerate with probability one).

\begin{thm}\label{thm:LC}
Suppose $\sM$ defined in \eqref{eq:Melnikov}
does not vanish identically.
Then the number of bifurcating limit cycles of the system \eqref{eq:pertcenter} is equal to the number of positive real zeros of $\sM$.
\end{thm}

For a proof of this result we direct the reader to a more general result that holds for perturbed Hamiltonian systems \cite[Sec.  2.1 of Part II]{ChLi} and \cite[Sec.  26]{IlyashenkoBook}.
In that setting, the function $\sM$ is generally an Abelian integral, but for the system \eqref{eq:pertcenter} it is simply a polynomial which can be seen from integrating \eqref{eq:Melnikov} in polar coordinates.
\begin{align*}
    \sM(r) &= \int_{x^2+y^2=r^2} p dy - q dx \\
    &= \int_{0}^{2\pi} p(r \cos(\theta),r \sin(\theta)) r\cos(\theta) d\theta +  q(r \cos(\theta),r \sin(\theta)) r\sin(\theta) d\theta \\
    &= \sqrt{8\pi}\sum_{m=0}^{\lfloor (d-1)/2 \rfloor} c_{m, n}\xi_m r^{2m+2},
\end{align*}
where
\begin{eqnarray} 
\sqrt{8\pi} c_{m, n} \xi_m &=&  \sum_{j+k=2m+1} \alpha_{j,k}\int_{0}^{2\pi}    (\cos(\theta))^{j+1}(\sin(\theta))^k   d\theta + \beta_{j,k}\int_{0}^{2\pi}    (\cos(\theta))^j(\sin(\theta))^{k+1}  d\theta.\nonumber\\
&=:& \sum_{j+k=2m+1} \alpha_{j,k} a_{k, m} + \beta_{j,k} a_{k+1, m}\label{eq:lincomb}.
\end{eqnarray}
The deterministic coefficients $a_{k, m}$ can be computed using the following identity \cite{Grad} (see also \cite{LC})
\begin{eqnarray}
a_{k, m}:=\int_{0}^{2\pi}    (\cos(\theta))^{2m+2-k}(\sin(\theta))^k   d\theta &=&
\begin{cases} 
2\pi \frac{(2m-k+1)!!(k-1)!!}{(2m+2)!!}, \quad \text{ for $k$ even} \\
0,  \quad \text{ for $k$ odd.}
\end{cases}\label{eq:a}
\end{eqnarray}

Hence, $\sM(r) = \sqrt{8\pi} r^2 f_n(r^2)$, with 
\begin{equation}
f_n(x)=\sum_{m=0}^n c_{m, n}\xi_m x^m\nonumber
\end{equation}
a polynomial of degree $n=\lfloor (d-1)/2 \rfloor $.
For $0\leq a < b$, the number of zeros of $\sM$ over the interval $(a,b)$ equals the number of zeros of $f_n$ over the interval $(a^2,b^2)$.
Together with Theorem \ref{thm:LC} and the Fundamental Theorem of Algebra, this implies that the number of bifurcating limit cycles of the system \eqref{eq:pertcenter} is at most $\lfloor (d-1)/2 \rfloor $ as stated in the introduction.

\subsection{
Polynomial perturbations with random coefficients}\label{sec:sigma}

Recall that the coefficients $\alpha_{j,k},\beta_{j,k}$ of the polynomials $p,q$ in the perturbed center \eqref{eq:pertcenter} are independent random variables. As stated in \eqref{eq:lincomb}, the resulting random coefficients $\xi_m$ of $f_n$ are linear combinations of $\alpha_{j,k}, \beta_{j,k}$.  Moreover, we notice that the linear combinations in \eqref{eq:lincomb}
involve disjoint subsets of indices for distinct values of $m$. This implies that $\xi_m$ are independent random variables with mean zero.

The formula \eqref{eq:a} allows for the explicit computation  \eqref{eq:sigma_exact} of the variances $c_{m, n}^2$  
\begin{equation}
c_{m, n}^2 = \frac{\pi}{2} \sum_{\ell=0}^m
\left(\frac{(2m-2\ell+1)!!(2\ell-1)!!}{(2m+2)!!}\right)^2 + 
 \left(\frac{(2m-2\ell-1)!!(2\ell+1)!!}{(2m+2)!!}\right)^2.\nonumber
\end{equation}
 
Let us verify the asymptotic $c_{m, n}^2 = m^{-1} + O(m^{-2})$ that was stated in \eqref{eq:sigma}.

For each $m$, the largest value (for $\ell$ ranging over $\ell=0,1,\dots,m$) of 
$$\left(\frac{(2m-2\ell+1)!!(2\ell-1)!!}{(2m+2)!!}\right)^2$$
is at $\ell=0$, and the next largest value is 
$$  \left(\frac{(2m-1)!!}{(2m+2)!!}\right)^2 $$
which occurs at $\ell=1$ and $\ell=m$.
The largest value of $$  \left(\frac{(2m-2\ell-1)!!(2\ell+1)!!}{(2m+2)!!}\right)^2 $$ is at $\ell=m$ and the next largest value 
is 
$$  \left(\frac{(2m-1)!!}{(2m+2)!!}\right)^2 $$
which occurs at $\ell = 0$ and $\ell=m-1$.
Hence, we can write
\begin{equation}\label{eq:sigmashort}
c_{m, n}^2 = \pi \left( \frac{(2m+1)!!}{(2m+2)!!} \right)^2 + \frac{\pi}{2}E_m,
\end{equation}
where $E_m$ is a sum of $2m$ non-negative numbers
each bounded by $\left(\frac{(2m-1)!!}{(2m+2)!!}\right)^2 $,
i.e., $E_m$ satisfies
$$ E_m \leq 2m \left(\frac{(2m-1)!!}{(2m+2)!!}\right)^2 .$$

By \cite[Remark 4.1]{LaNa},
$$ \frac{(2n-1)!!}{(2n-2)!!} = \frac{2}{\sqrt{\pi}} n^{1/2} (1+O(n^{-1}))$$
which implies
\begin{equation}
\pi \left( \frac{(2m+1)!!}{(2m+2)!!}\right)^2 = m^{-1} (1+O(m^{-1})),\nonumber
\end{equation}
and
\begin{equation}
\pi \left( \frac{(2m-1)!!}{(2m+2)!!}\right)^2 = \frac{1}{4}m^{-3} (1+O(m^{-1})).\nonumber
\end{equation}
Applying this in \eqref{eq:sigmashort} gives
$$c_{m, n}^2 = m^{-1} + O(m^{-2}), $$
as desired.

\subsection{Proofs of the Corollaries}

\begin{proof}[Proof of Corollary \ref{cor:limit:cycle}]
As in the statement of the corollary, suppose $p,q$ have independent coefficients,
and as above let $f_n$ be the polynomial such that the Melnikov integral can be expressed as $\sM(r) = r^2f(r^2)$.  As proven in Section \ref{sec:sigma}, the resulting random coefficients $c_{m, n}\xi_m$ are independent with variance $c_{m, n}^2$ given by \eqref{eq:sigma_exact}.
In light of the asymptotic \eqref{eq:sigma} for the variance $c_{m, n}^2$ of the coefficients, the result then follows from Theorem \ref{thm:main:critical} and Theorem \ref{thm:LC} if we can show that the $\xi_m$ have uniformly bounded $(2+\ep_0)$-moments. By Minkowski's inequality, we have from \eqref{eq:lincomb} that
\begin{eqnarray*}
\left (\E |c_{m, n} \xi_m|^{2+\ep_0}\right )^{1/(2+\ep_0)} &\le& \sum_{j+k = 2m+1} a_{k, m} \left (\E |\alpha_{j, k}|^{2+\ep_0}\right )^{1/(2+\ep_0)} + a_{k+1, m}\left (\E |\beta_{j, k}|^{2+\ep_0}\right )^{1/(2+\ep_0)} \\
&\ll&   \sum_{k=0}^{2m+2} a_{k, m} = \sum_{\ell=0}^{m+1} a_{2l, m} \ll \sum_{\ell=0}^{m+1}  \frac{(2m-2\ell+1)!!(2\ell-1)!!}{(2m+2)!!} \\
&\ll&    \frac{(2m+1)!!}{(2m+2)!!}   + m  \frac{(2m-1)!!}{(2m+2)!!}  \ll  m^{-1/2}
\end{eqnarray*}
where on the second line, we used \eqref{eq:a}, and on the third line, we used the same argument that led to \eqref{eq:sigmashort}. Using $|c_{m, n}| = \Theta(m^{-1/2})$, we conclude that the $\xi_m$ have uniform bounded $(2+\ep_0)$-moments.
\end{proof}

\begin{proof}[Proof of Corollary \ref{cor:Lienard}]
In order to simplify the final step, we include a constant factor in the function $p(x)$ as follows
\begin{equation}\label{eq:gscale}
p(x) = \frac{1}{2\sqrt{\pi}} \sum_{k=1}^d \alpha_k x^k.
\end{equation}
The Poincar\'e-Pontryagin-Melnikov integral associated to the system  \eqref{eq:Lienard} is
\begin{align*}
    \sM(r) &= \int_{x^2+y^2=r^2} p(x) dy = \int_{0}^{2\pi} p(r \cos(\theta)) r\cos(\theta) d\theta  = \sum_{m=0}^{(d-1)/2} c_{m, n} \xi_m r^{2m+2},
\end{align*}
where
\begin{align*}
    c_{m, n} \xi_m &=  \frac{\alpha_{m}}{2\sqrt{\pi}}\int_{0}^{2\pi}    (\cos(\theta))^{2m+2}  d\theta  =  \alpha_{m} \sqrt{\pi}  
\frac{(2m+1)!!}{(2m+2)!!}.
\end{align*}
The variance of the left-hand side satisfies
\begin{equation}
c_{m, n}^2 = \pi
\left(\frac{(2m+1)!!}{(2m+2)!!}\right)^2= m^{-1} + O(m^{-2}).\nonumber
\end{equation}
Similarly to the proof of Corollary \ref{cor:limit:cycle}, the $\xi_m$ have uniformly bounded $(2+\ep_0)$-moments. Therefore, the stated result follows from Theorem \ref{thm:main:critical}.
\end{proof}

\section{Proof of Theorem \ref{thm:main:critical}} \label{sec:pf:main}
  Throughout this section, $N_n(S)=N_{f_n}(S)$ denotes the number of roots of $f_n$ in a set $S$.  
\subsection{Splitting into pieces}
Since the negation of any root of $f_n$ in $(-\infty, 0)$ is a root of the polynomial
$$g_n(x) = f_n(-x) = \sum_{m=0}^n c_{m, n} ((-1)^{m}\xi_m) x^m$$
and since the random variables $((-1)^{m}\xi_m)_m$ also satisfy the conditions of Theorem \ref{thm:main:critical}, any bound for the roots of $g_n$ in $(0, \infty)$ gives the same bound for $f_n$ in $(-\infty, 0)$. Therefore, it suffices to study the nonnegative roots only. In particular, it suffices to show the following lemmas. 
\begin{lemma} [Roots in $(0, 1)$] \label{lm:-1/2:01}
Under the hypothesis of Theorem \ref{thm:main:critical}, we have
	 \begin{equation}\label{eq:01:uni:2}
	 \E N_{n}[0, 1] = o(\log n).
	 \end{equation}
\end{lemma}
\begin{lemma}[Roots in $(1, \infty)$]\label{lm:allrho:1}
		 Let $\rho$ be a real number. Let $\ep_0$, $C_{2+\ep_0}$ be positive constants and $(a_m)$ be a sequence converging to $0$. Let
		 $$f_n(x) = \sum_{m=0}^n c_{m, n} \xi_m x^m, $$
		 be a random polynomial where $\xi_m$ are independent  random variables satisfying Assumption-A with parameters $\ep_0$ and $C_{2+\ep_0}$. Assume that the deterministic coefficients $(c_{m, n})_{0\le m\le n}$ satisfy Assumption-B with parameters $\rho$ and $(a_m)$. We have
		 \begin{equation}
	\E N_n(1,\infty) = \frac{1+o(1)}{2\pi}\log n.\nonumber
	\end{equation}
\end{lemma}
We note that Lemma \ref{lm:allrho:1} holds for all $\rho$ (rather than just for $\rho =  -1/2$), which may be of independent interest.

For the proofs, we use the universality approach which roughly speaking reduces the general distribution of $\xi_m$ to the case that $\xi_m$ are iid standard Gaussian. We first start with a reduction to a small neighborhood around 1.
\subsection{A reduction to the core interval} As typical for the generalized Kac family with $\rho>-1/2$, the real roots concentrate near $\pm 1$  (see for example, \cite{kabluchko2014asymptotic, DOV}). In this section, we extend this for $\rho\le -1/2$.

Let 
\begin{equation}\label{def:In}
\I_n = \left [1 - \exp\left (-(\log n)^{1/5}\right ), 1-\exp\left ((\log n)^{1/5}\right )/n\right ].
\end{equation}
The choice of $\I_n$ is not strict; basically, we need $\I_n$ to be of the form $[1-\ep_n, 1 - \frac{1}{\ep'_n n}]$ where $\ep_n$ and $\ep_n'$ go to $0$ faster than $\frac{1}{\log n}$ but not too fast.
The following result shows that for generalized Kac polynomials with growth power larger than $-1/2$, most of the roots stay inside $\I_n$.
 \begin{prop}\cite[Proposition 2.3]{nguyenvuCLT}\label{prop:main interval:0} Let $p_n = \sum_{m=0}^{n} a_m\xi_m x^{m}$ where $a_m = \Theta(m^{\tau})$ for some $\tau>-1/2$. Then 
	\begin{equation}
	\E N_{p_n}^{2}([0, 1]\setminus \I_n) \ll (\log n)^{4/5}  \quad\text{and}\quad \E N_{p_n}^{2}([1, \infty)\setminus \I_n^{-1}) \ll (\log n)^{4/5}\nonumber
	\end{equation}
	where $\I_n^{-1} = \{x^{-1}: x\in \I_n\}$.
\end{prop}
Extending this proposition, we get
\begin{lemma}\label{lm:main interval} Proposition \ref{prop:main interval:0} holds for all $\rho\in \R$. In particular, we also have
	\begin{equation}
	\E N_{f_n} ([0, \infty)\setminus (\I_n\cup \I_n^{-1})) \ll (\log n)^{2/5} =o(\sqrt{\log n}).\label{eq:bound:outside}
	\end{equation}
\end{lemma}
Before proving this lemma, we need a simple observation.
\begin{lemma} \label{lm:interlace}
	Let $P$ be a polynomial and $J$ be an interval on the real line. Let $k$ be any positive integer. We have
	$$N_P(J)\le k+ N_{P^{(k)}}(J).$$
\end{lemma}
\begin{proof}
	Since between any two roots of $P$, there must be at least one root of $P'$, we get
	$$N_P(J)\le 1+ N_{P'}(J).$$
	Repeating this inequality $k$ times, we obtain the claim.
\end{proof}

	\begin{proof}[Proof of Lemma \ref{lm:main interval}]  
	Since the set $([0, \infty)\setminus (\I_n\cup \I_n^{-1}))$ is a union of three intervals, we can apply Lemma \ref{lm:interlace} to get for any integer $k\ge 0$ that
	$$N_{f_n} ([0, \infty)\setminus (\I_n\cup \I_n^{-1}))\le 3k+N_{f^{(k)}_n } ([0, \infty)\setminus (\I_n\cup \I_n^{-1})).$$
 Since $f^{(k)}_n = \sum_{m=0}^{n-k} \Theta(m^{\rho+k}) \xi_{m+k} x^{m}$,  by choosing $k = \lceil -\rho\rceil$, $f^{(k)}_n$ satisfies the hypothesis of Proposition \ref{prop:main interval:0} with $\tau = \rho+k\ge 0$, and so,
	$$\E N^{2}_{f_n} ([0, \infty)\setminus (\I_n\cup \I_n^{-1})) \ll 1+(\log n)^{4/5} \ll (\log n)^{4/5}.$$
	Applying this bound and Holder's inequality, we get \eqref{eq:bound:outside}.
\end{proof}

\subsection{Universality theorem} In view of Lemmas \ref{lm:-1/2:01} and \ref{lm:allrho:1}, the result in Lemma \ref{lm:main interval} allows us to restrict to the main interval $\I_n$. To handle the roots here, we shall use the so-called universality method -- reducing the general case to the case that $\xi_i$ are iid Gaussian. We first state several conditions that are sufficient for universality (for a first reading, we recommend skipping these conditions and go straight to Theorem \ref{thm:universality}). Consider the polynomials
$$p_n = \sum_{m=0}^{n} a_m \xi_m x^{m}\quad\text{and}\quad \tilde p_n = \sum_{m=0}^{n} a_m \tilde \xi_m x^{m}$$
where $\xi_m$ and $\tilde \xi_m$ are independent random variables satisfying Assumption-A with some positive constants $\ep_0, C_{2+\ep_0}$ - in practice, we take $\tilde \xi_m$ to be standard Gaussian. Let us look at a local interval $[x_0-r, x_0+r]$ for some $x_0\in \R$, $r>0$, and an error factor $\delta\in (0, 1)$. Fix some positive constants $\alpha_1$ and $C_1$. We say that $p_n$ satisfies the universality condition with parameters $x_0, r, \delta$ (which may depend on $n$) and constants $\alpha_1, C_1$ (independent of $n$) if there exists a constant $C$ such that the following hold. Let $A = 6(C_1+2) + \frac{\alpha_1\ep_0 }{60}$ and $c_0 = \frac{\alpha_1\ep_0}{10^9}$ (these choices of $A$ and $c_0$ are for specification only. In practice, these conditions often hold for any choice of $A>0$ and $c_0>0$).

{\bf Universality Conditions.}
\begin{enumerate}
	\item {\it Delocalization:}\label{cond-delocal} For every $z\in B (x_0, r)$, it holds for all $m = 0, \dots, n$ that
	$$\frac{|a_m||z|^{m}}{\sqrt{\sum _{j = 0}^{n}|a_j|^{2}|z|^{2j}}}\le C \delta^{\alpha_1}.$$ 
	
	\item {\it Derivative growth:}\label{cond-repulsion} For any real number $y\in [x_0-r, x_0+r]$,
	\begin{equation}
	\sum_{m=0}^{n} m^{2}|a_m|^{2}|y|^{2m-2}\le C r^{-2} \delta^{-c_0}\sum_{m=0}^{n}|a_m|^{2}|y|^{2m},\nonumber
	\end{equation}
	\begin{equation}
	\sum_{m=0}^{n} \sup_{z\in B(y, r/2)} m^{4}|a_m|^{2}|z|^{2m-4}\le C r^{-4}\delta^{-c_0}\sum_{m=0}^{n} |a_m|^{2}|y|^{2m}.\nonumber
	\end{equation}

	\item {\it Anti-concentration:}\label{cond-smallball} With probability at least $1 - C\delta^{A}$, there exists $x'\in B(x_0, r/100)$ for which $|f_n(x')|\ge \exp(-\delta^{-c_0})$.

	\item {\it Boundedness:} \label{cond-bddn} With probability at least $1 - C \delta^{A}$, $|p_n(w)|\le \exp(\delta^{-c_0})$ for all $w\in B (x_0, r/16)$.

	\item \label{cond-poly} It holds that
	$$\E N^{3}\textbf{1}_{N\ge \delta^{-C_1}} \le C,$$
	where $N$ is the number of zeros of $p_n$ in the disk $B(x_0, r/20)$. We note that throughout this paper, if $p_n$ is identically 0, we adopt the (admittedly artificial) convention that $p_n$ has no roots in $\C$.
\end{enumerate}

Here, $B(x_0, r)$ denotes the closed ball on the complex plane with center $x_0$ and radius $r$. The following theorem is a simplified version of \cite[Theorem 2.6]{nguyenvurandomfunction17}. 
\begin{thm}[Comparison]\label{thm:universality} Assume that $\xi_i$ and $\tilde \xi_i$ are independent real-valued random variables with mean 0 and variance 1. Assume that there exist positive constants $\ep_0, C_{2+\ep_0}$ such that $\E |\xi_m|^{2+\ep_0}<C_{2+\ep_0}$ and $\E |\tilde \xi_m|^{2+\ep_0}<C_{2+\ep_0}$ for all $m$. 
	Consider the random polynomials
	$$p_n = \sum_{m=0}^{n} c_{m, n} \xi_m x^{m}\quad\text{and}\quad \tilde p_n = \sum_{m=0}^{n} c_{m, n} \tilde \xi_m x^{m}$$
	where the $c_{m, n}$ are deterministic coefficients that may depend on $n$.
	Assume furthermore that there exist constants $\alpha_1, C_1$ such that the Universality Conditions \eqref{cond-delocal}--\eqref{cond-poly} hold for both $p_n$ and $\tilde p_n$ with parameters $x_0, r, \delta$ and constants $\alpha_1, C_1$. 
	Then there exist positive constants $C', c$ depending only on the constants $\ep_0, C_{2+\ep_0}, \alpha_1, C_1, C$ (but not on the parameters $x_0, r, \delta$ and $n$) such that the following holds.
	
	For any function $G: \R \to \R$ supported on $ [x_0-10^{-5}r, x_0+10^{-5} r] $ with continuous derivatives up to order $6$ and $||G^{(a)}||_{\infty}\le r^a$ for all $0\le a\le 6$, we have
	\begin{eqnarray} 
	\left |\E\sum_{\zeta:\  p_n(\zeta) =0} G\left (\zeta  \right) 
	-\E\sum_{\tilde \zeta: \ \tilde p_n(\tilde \zeta) =0} G\left (\tilde \zeta  \right) \right |\le C' \delta^{c}.\nonumber
	\end{eqnarray}
	
\end{thm}
 In fact, the result in \cite{nguyenvurandomfunction17} is stronger, involving high-degree correlation functions of the roots. Since we focus on the number of real roots here, we only need this version.

\subsection{Roots inside $[0, 1]$} Toward the proof of Lemma \ref{lm:-1/2:01}, we now use the universality theorem to reduce the task of bounding $\E N_{f_n}(\I_n)$ to bounding $\E N_{\tilde f_n}(\I_n)$ where 
$$\tilde f_n = \sum_{m=0}^{n} c_{m, n} \tilde \xi_m x^{m}$$
with $\tilde \xi_m$ being iid standard Gaussian.
The goal of this subsection is to show the following.
 \begin{lemma} [Reduction to the Gaussian case inside $(0,1)$]\label{lm:reduction:01} Under the hypothesis of Lemma \ref{lm:-1/2:01}, there exist positive constants $C, c'$ depending only on $\ep_0$, $C_{2+\ep_0}$ and the rate of convergence of $\sqrt m |c_{m, n}|$   such that the following holds for every $x_0\in \I_n$. Setting
 	$r = \frac{1-x_0}{3}$ and $\delta = (\log (r^{-1}))^{-1}$, for every function $G: \R \to \R$ supported on $ [x_0-10^{-5}r, x_0+10^{-5}r] $ with continuous derivatives up to order $6$ and $||G^{(a)}||_{\infty}\le r^a$ for all $0\le a\le 6$, we have
 	 \begin{eqnarray} 
 	 \left |\E\sum_{\zeta:\ f_n(\zeta) =0} G\left (\zeta  \right) 
 	 -\E\sum_{\tilde \zeta:\ \tilde f_n(\tilde \zeta) =0} G\left (\tilde \zeta  \right) \right |\le C' \delta^{c}\nonumber
 	 \end{eqnarray}
 \end{lemma}
By Theorem \ref{thm:universality}, in order to prove Lemma \ref{lm:reduction:01}, it suffices to show that there exist constants $\alpha_1, C_1$ such that Universality Conditions \eqref{cond-delocal}--\eqref{cond-poly} hold for all $x_0\in \I_n$ with the corresponding parameters  $r = \frac{1-x_0}{3}$ and $\delta = (\log (r^{-1}))^{-1}$. For notational convenience, we let 
$$L = r^{-1}.$$
We have $\delta = (\log L)^{-1}$ and $L$ goes to $\infty$ with $n$ for $x_0\in \I_n$. In particular, $L\ge \exp((\log n)^{1/5})$.
\subsubsection{Universality Condition \eqref{cond-delocal}} For all $z\in B(x_0, r)$, we have $|z|= 1-\frac{1}{\Theta(L)}$. So, $|c_i||z|^{i}\ll 1$ while 
 	 \begin{eqnarray}
 	 \sum _{m = 0}^{n}|c_{m, n}|^{2}|z|^{2m} \ge  \sum _{m = 0}^{L}|c_{m, n}|^{2}|z|^{2m} \gg \sum _{m = 0}^{L} |c_{m, n}|^{2} \gg \sum _{m = 0}^{L}\frac{1}{ m} \gg \log L.\label{eq:var:f}
 	 \end{eqnarray}
 	 Thus, Condition \eqref{cond-delocal} holds for $\alpha_1 = 1/2$.

  \subsubsection{Universality Condition \eqref{cond-repulsion}}
  
  For the first inequality, we observe that for all $y\in [x_0-r, x_0+r]$,
  \begin{eqnarray}
  \sum_{m=0}^{n} m^{2}|c_{m, n}|^{2}|y|^{2m-2} \ll \sum_{m=0}^{n} m \left ( 1 - \frac{1}{\Theta(L)}\right )^{m} \ll L^{2}\ll  L^{2}\sum _{m = 0}^{n}|c_{m, n}|^{2}|y|^{2m}
  \end{eqnarray}
  where we used \eqref{eq:var:f} for $y$ in place of $z$. 
  For the second inequality, we have
  \begin{equation}
  \sum_{m=0}^{n} \sup_{z\in B(y, r)} m^{4}|c_{m, n}|^{2}|z|^{2m-4}\ll \sum_{m=0}^{n}   m^{3}\left ( 1 - \frac{1}{\Theta(L)}\right )^{m}\ll L^{4}\ll L^{4} \sum_{m=0}^{n} |c_{m, n}|^{2}|y|^{2m}.\nonumber
  \end{equation}
  Thus, Condition \eqref{cond-repulsion} holds for any positive constant $c_0$.
  \subsubsection{Universality Condition \eqref{cond-smallball}} We shall prove that this condition holds for any $A>0$ and $c_0>0$.
  We shall use the following lemma.
  \begin{lemma}\label{lmanti_concentration} \cite[Lemma 9.2]{nguyenvurandomfunction17}
  	Let $\mathcal E$ be an index set of size $N\in \N$, and let $(\xi_j)_{j\in \mathcal E}$ be independent random variables with variance 1, and bounded $(2+\ep_0)$-moments uniformly bounded by some constant $C_{2+\ep_0}$. Let $(e_j)_{j\in \mathcal E}$ be deterministic coefficients with $|e_j|\ge \bar e$ for all $j$ and for some number $\bar e$. Then for any $B\ge 1$, any interval $I\subset\R$ of length at least $N^{-B}$, there exists an $x\in I$ such that
  	$$\sup_{Z\in \R}\P\left (\left| \sum_{j\in \mathcal E} e_j \xi_j \cos(jx)-Z\right |\le \bar e N^{-16B^{2}}\right )= O\left (N^{-B/2}\right )$$
  	where the implicit constant only depends on $B, \ep_0, C_{2+\ep_0}$.
  \end{lemma} 
  
  Back to our proof of the Condition \eqref{cond-smallball}. It follows from a more general anti-concentration bound: there exists $\theta\in I := [-r/100,  r/100]$ such that 
  \begin{equation}
  \sup _{Z\in \C}\P\left (|f_n(x_0 e^{\sqrt{-1}\theta})-Z|\le \exp(-(\log L)^{c_0})\right )\ll (\log L)^{-A}.\nonumber
  \end{equation}
  
  Since the probability of being confined in a complex ball is bounded from above by the probability of its real part being confined in the corresponding interval on the real line, it suffices to show that
  \begin{equation}
  \sup _{Z\in \R}\P\left (\left |\sum_{j=0}^{n} c_{j, n} \xi_j x_0^{j} \cos{(j\theta)}-Z\right |\le \exp(-(\log L)^{c_0})\right )\ll (\log L)^{-A}.\nonumber
  \end{equation}
  
  For an $N\le n$ to be chosen, by conditioning on the random variables $(\xi_j)_{j\ge N}$, the above bound is deduced from 
  \begin{equation}
  \sup _{Z\in \R}\P\left (\left |\sum_{j=0}^{N-1} c_{j, n} \xi_j x_0^{j} \cos{(j\theta)}-Z\right |\le \exp(-(\log L)^{c_0})\right )\ll (\log L)^{-A}.\nonumber
  \end{equation}
  This is, in turn, a direct application of Lemma \ref{lmanti_concentration} with $B = 2A, \mathcal E = \{0, \dots, N-1\}, N = \log L$ and $\bar e = \min\{|c_{j, n}| x_0^j, j\le N-1\} \gg (\log L)^{-1/2}$ and the observation that
  \begin{equation}\label{key}
  \exp(-(\log L)^{c_0}) = o( (\log L)^{-16B^{2}-1/2}).\nonumber
  \end{equation}
  This verifies Condition \ref{cond-smallball}.
  
  \begin{remark}
  	We remark that by using the same argument as above with $N =  K\log L$, one can show that there exists a constant $C$ (depending only on $A, \ep_0, C_{2+\ep_0}, (a_m)$) such that for every $1\le K\le n/\log L$,  there exists $x'= x_0 e^{\sqrt{-1}\theta}$ where $\theta\in [-r/100, r/100]$ for which
  	\begin{equation}
  	\P\bigg (|f_n(x')|\le \exp\big(-(\log L)^{c_0} -C K\big)\bigg )\le \frac{C}{K^{A}(\log L)^{A}}.\label{smallball_kac}
  	\end{equation}
  	This shall be useful later.
  \end{remark}

  \subsubsection{ Universality Condition \eqref{cond-bddn}} \label{sec:cond:bddn:0}
We shall prove that condition \eqref{cond-bddn} holds for any $A>0$.
  \begin{lemma}\label{lm:bddn}
  	Universality Condition \eqref{cond-bddn} holds. Moreover, there exists a constant $C$ such that for all $x_0\in \I_n$ and all $\alpha\in (0, 1)$ that may depend on $n$, it holds that
  	\begin{equation}\label{eq:cond:bddn}
  	\P\left (\sup_{z\in B (x_0, r/16)} |f_n(z)|\le C\alpha^{-1/2}\right )\ge 1 - C\alpha \log L
  	\end{equation}
  \end{lemma}
  Before proving this lemma, we establish related bounds on the derivatives of $f_n$ at a single point $x_0$.
  \begin{lemma} \label{lm:net}
  	There exists a constant $C$ such that the following holds for any $\alpha$ in $(0, 1)$ and any integer $i\in [0, \log L]$: 
  	\begin{equation}\label{eq:lm:net}
  	\P\left (|f_n^{(i)}(x_0)|\le \alpha^{-1/2}\sqrt{(2i-1)!} (4L)^{i} \right )\ge 1 - C\alpha.
  	\end{equation}
  \end{lemma}
  
  We also need a crude bound on the $(\log L)$-th derivative to control the error term in the Taylor expansion of $f_n$ around $x_0$.
  \begin{lemma} \label{lm:m:derivative}
  	There exists a constant $C$ such that the following holds for any $\alpha$ in $(0, 1)$ and for $m = \lfloor \log L\rfloor$:
  	\begin{equation}\label{eq:lm:m}
  	\P\left (\sup_{w\in B(x_0, r)}|f_n^{(m)}(w)|\le  \alpha^{-1/2} m! (2L)^{m+1}\right )\ge 1 - C\alpha.
  	\end{equation}
  \end{lemma}
  Deferring the proofs of these lemmas, we proceed to proving Lemma \ref{lm:bddn}.
  \begin{proof}[Proof of Lemma \ref{lm:bddn}]
  	For a parameter $\alpha\in (0, 1)$ to be chosen. Let $\mathcal A$ be the intersection of events in Equation \eqref{eq:lm:net} for all integers $i\in [0, \log L]$ together with the event in Equation \eqref{eq:lm:m}.
  	By the union bound, the event $\mathcal A$ has probability at least $1 - O(\alpha \log L)$. Under the event $\mathcal A$, we show that $|f_n(z)|\ll \alpha^{-1/2}$ for all $z\in B (x_0, r/16))$. 
  	
  	Indeed, let $m = \lfloor \log L\rfloor$, we apply $(m-1)$-th Taylor expansion of the function $f_n$ around $x_0$ and obtain   
  	\begin{eqnarray*}
  		|f_n(z)| &\le& \sum_{i=0}^{m-1} \frac{|f_n^{(i)}(x_0)| |z-x_0|^{i}}{i!} +  \frac{\sup_{y\in B(x_0, 1/(2L))} |f_n^{m}(y)||z-x_0|^{m}}{m!}\\
  		&\le& \sum_{i=0}^{m-1} \frac{\alpha^{-1/2} \sqrt{(2i-1)!} (4L)^{i} |z-x_0|^{i}}{i!} +  \frac{\alpha^{-1/2} m! (2L)^{m+1} |z-x_0|^{m}}{m!}\\
  		&\ll& \alpha^{-1/2}\sum_{i=0}^{m-1} (8|z-x_0| L)^{i} + \alpha^{-1/2} L(2|z-x_0| L)^{m}\ll \alpha^{-1/2}
  	\end{eqnarray*}
  	where in the last inequality, we used $|z-x_0|\le 1/(16L)$ for the first term and the choice $m=\lfloor \log L\rfloor$ to swallow the extra factor $L$ on the second term.
  	Thus, 
  	$$\P\left (\sup_{z\in B (x_0, r/8)} |f_n(z)|\le C\alpha^{-1/2}\right )\ge \P(\mathcal A)\ge 1 - C\alpha \log L$$
  	giving \eqref{eq:cond:bddn}. Setting $\alpha = \exp(-(\log L)^{c_0})$, we get Condition \eqref{cond-bddn}.
  \end{proof}
  
  \begin{proof} [Proof of Lemma \ref{lm:net}]
  	By Chebyshev inequality, it suffices to show that the variance of $f_n^{(i)}(x_0)$ is bounded by $(2i-1)! (4L)^{2i}$ (recall that $\E f_n^{(i)}(x_0)=0$). Indeed, we have
  	\begin{eqnarray}
  	f_n^{(i)}(x_0) = \sum_{k=i}^{n} k (k-1)\dots (k-i+1) c_{k, n} \xi_k x_0^{k-i}.\nonumber
  	\end{eqnarray}
  	Thus, 
  	\begin{eqnarray}
  	&&	\Var f_n^{(i)}(x_0) = \sum_{k=i}^{n} k^{2} (k-1)^{2}\dots (k-i+1)^{2} c_{k, n}^{2} x_0^{2k-2i}\nonumber\\
  	&\ll &\sum_{k=i}^{n} k (k-1)^{2}\dots (k-i+1)^{2}  x_0^{2k-2i} = \sum_{k=i}^{3i-1} + \sum_{k=3i}^{n} =: S_1+S_2\nonumber.
  	\end{eqnarray}
  	For the main term $S_2$, we note that for all $\ell \ge 2i$, we have $\ell\le 2(\ell -i+1)$. Thus, for all $k\ge 3i$, we have
  	$$(k-1) \dots (k-i+1) \le 2^{i} (k-i)\dots (k-2i+2).$$
  	So, 
  	\begin{eqnarray*}
  		S_2 &\le& 2^{i}\sum_{k=3i}^{n} k (k-1) \dots (k-2i+2)   x_0^{2k-2i} \ll 2^{i}\sum_{k=0}^{\infty} k (k-1) \dots (k-2i+2)   y^{k-2i+1}\\
  		&&\quad\text{(where $y = x_0^{2}\le 1 - \frac{1}{2L}$)}\\
  		&=& 2^{i}\left (\sum_{k=0}^{\infty} y^{k}\right )^{(2i-1)} = 2^{i} \left (\frac{1}{1-y}\right )^{(2i-1)} =  \frac{2^{i} (2i-1)!}{(1-y)^{2i}}\ll   (2i-1)! (4L)^{2i}.
  	\end{eqnarray*}

  	For $S_1$, we have the crude bound
  	\begin{eqnarray}
  	S_1\le \sum_{k=i}^{3i-1} (3i)^{2i-1} \le (3i)^{2i}\ll  (2i-1)! (4L)^{2i}.\nonumber 
  	\end{eqnarray}
  	So, $\Var  f_n^{(i)}(x_0)\ll   (2i-1)! (4L)^{2i}$. This completes the proof by applying Chebyshev inequality.
  \end{proof}
  
  \begin{proof} [Proof of Lemma \ref{lm:m:derivative}]
  	Consider the event that $|\xi_k|\le \sqrt k \alpha^{-1/2}$ for all $k\in [m, n]$. By the assumption that $\xi_k$ has uniformly bounded $(2+\ep_0)$-moments, this event holds with probability at least  
  	$$ 1-O(1)\sum_{k=m}^{n}  k^{-(2+\ep_0)/2}\alpha^{(2+\ep_0)/2}  =1-  O(\alpha).$$
  	
  	Under this event, we have 
  	\begin{eqnarray*}
  		\sup_{w\in B(x_0, r)}|f_n^{(m)}(w)| &\le& \sum_{k=m}^{n} k (k-1) \dots (k-m+1) |c_{k, n}| |\xi_k| \left (1-  \frac{1}{2L}\right )^{k-m}\\
  		&\ll& \alpha^{-1/2}\sum_{k=m}^{n} k (k-1) \dots (k-m+1) \left (1-  \frac{1}{2L}\right )^{k-m}\\
  		&\le& \alpha^{-1/2} \left (\frac{1}{1-z}\right )^{(m)}\bigg\vert _{z = 1 - 1/(2L)} =  \alpha^{-1/2}  m! (2L)^{m+1}
  	\end{eqnarray*}
  	as claimed.
  \end{proof}

  \subsubsection{Universality Condition \eqref{cond-poly}} We show that this condition holds with $C_1 = 1$.
  For this proof, we shall need the classical Jensen's bound for the number of roots of analytic functions. Assume that $f$ is an analytic function on an open domain that contains a closed ball $B(z, R)$. Then for any $R'<R$, we have
  \begin{equation}
  N_f(B(z, R'))\le  \frac{\log \frac{M}{M'}}{\log\frac{R^{2}+R'^{2}}{2RR'}}\label{jensenbound}
  \end{equation}
  where $M = \max_{w\in B(z, R)} |f(w)|$ and $M' = \max_{w\in B(z, R')} |f(w)|$. For a proof, we refer to \cite[Appendix 15.5]{nguyenvurandomfunction17}.

  	Applying Jensen's bound \eqref{jensenbound} to $f_n$ with $R = r/16$ and $R' = r/20$, we obtain
  	\begin{equation}
  	N\le \frac{\log \frac{M}{M'}}{\log\frac{R^{2}+R'^{2}}{2RR'}} \ll \log \frac{M}{M'}.
  	\end{equation}
  	where $N = N_{f_n}(B(x_0, R'))$, $M = \max_{w\in B(x_0, R)} |f_n(w)|$ and $M' = \max_{w\in B(x_0, R')} |f_n(w)|$.
  	From \eqref{smallball_kac}, for all $1\le K\le n/\log L$, we have
  	\begin{equation}
  	\P\bigg (M' \le \exp\big(-(\log L)^{c_0} -C K\big)\bigg )\ll \frac{1}{K^{A}(\log L)^{A}}.\nonumber
  	\end{equation}
  	From \eqref{eq:cond:bddn}, we have for all $\alpha\in (0, 1)$,
  	\begin{equation}
  	\P\left (M\le C\alpha^{-1/2}\right )\ge 1 - C\alpha \log L.\nonumber
  	\end{equation}
  	Hence, for all $1\le K\le n/\log L$,
  	\begin{equation}
  	\P\bigg (M \ge \exp\big((\log L)^{c_0} +C K\big)\bigg )\ll \frac{1}{K^{A}(\log L)^{A}}.\nonumber
  	\end{equation}
  	Combining these inequalities, we get for every $1\le K\le n/\log L$ that
  	\begin{equation}\label{eq:bound N}
  	\P(N\gg (\log L)^{c_0} + CK) \ll \frac{1}{K^{A}(\log L)^{A}}.
  	\end{equation}
  	
  	Thus, let $k_0$ be the largest integer such that $e^{k_0} (\log L)^{c_0}\le n/\log L$, we have
  	\begin{eqnarray*}
  		\E N^{6} &\le& (\log L)^{6c_0} + \sum_{k=0}^{k_0} \E N^{6} \textbf{1}_{e^{k}(\log L)^{c_0}\le N\le e^{k+1} (\log L)^{c_0}} + \E N^{6} \textbf{1}_{N\ge e^{k_0+1} (\log L)^{c_0}}\\
  		&\ll& (\log L)^{6c_0} + \sum_{k=0}^{k_0}  e^{6k} (\log L)^{6c_0} e^{-Ak}  (\log L)^{-A} + n^{6} e^{-Ak_0}  (\log L)^{-A(1+c_0)} \text{ by \eqref{eq:bound N}}\\
  		&\ll& (\log L)^{6c_0} 
  	\end{eqnarray*}
  	where in the second line, we used $N\le n$ a.s. So, by H\"{o}lder's inequality and another application of \eqref{eq:bound N},
  	\begin{eqnarray*}
  		\E N^{3}\textbf{1}_{N\ge \log L} &\le& \left ( \E N^{6}\right )^{1/2} \P(N\ge  \log L )^{1/2}\ll (\log L)^{3c_0}   (\log L)^{-A/2}\ll 1.
  	\end{eqnarray*}
  	This proves Condition \eqref{cond-poly}  and hence completes the proof of Lemma \ref{lm:reduction:01}.

  \subsubsection{Proof of Lemma \ref{lm:-1/2:01}}\label{sec:proof:-1/2:01}
  Having proved Lemma \ref{lm:reduction:01} that reduces to the Gaussian case for smooth test functions $G$, we are now ready to pass this result to the number of real roots and prove Lemma \ref{lm:-1/2:01}.  We shall show that there exists a positive constant $\alpha$ such that for every $x_0\in \I_n$, 
  \begin{equation}\label{eq:uni:ex}
  \E N_{f_n}\left (a, b\right ) - \E  N_{\tilde f_n}\left (a, b\right ) \ll \left (\log \frac{1}{1-x_0}\right )^{-\alpha}
  \end{equation}
where $a = x_0-\frac{1-x_0}{10^6}$ and $b=x_0+\frac{1-x_0}{10^6}$.
In fact, we shall see from the proof that this inequality also holds if on the LHS, one replaces the interval $\left (a, b\right ) $ by any of its subintervals. Assuming this and decomposing $\I_n$ into such intervals, we obtain
  \begin{equation}\label{eq:uni:In:f}
  \E N_{f_n}(\I_n)- \E N_{\tilde f_n}(\I_n) \ll \sum_{k=(\log n)^{1/5}}^{C \log n } k^{-\alpha} \ll (\log n)^{1-  \alpha} = o(\log n).
  \end{equation}
  In Lemma \ref{lm:main interval}, we showed that the number of roots in $[0, 1]\setminus \I_n$ of $f_n$ (and also $\tilde f_n$ as it is just a special case of $f_n$) is $o(\log n)$. Combining this with \eqref{eq:uni:In:f} yields
  \begin{equation} 
  \E N_{f_n}[0, 1]- \E N_{\tilde f_n}[0, 1]  = o(\log n).\nonumber
  \end{equation}
  In Section \ref{sec:01:Gau}, we shall show that 
  \begin{equation}\label{eq:Gauss:01}
  \E N_{\tilde f_n}[0, 1]  = o(\log n)
  \end{equation}
  and so this completes the proof of Lemma \ref{lm:-1/2:01}.

   \begin{proof}[Proof of \eqref{eq:uni:ex}]
  	As before, let $L = 3/(1-x_0)$. Let $G$ be a smooth function on $\R$ with support in   $\left [a-L^{-1}(\log L)^{-\alpha}, b+L^{-1}(\log L)^{-\alpha}\right] $ such that $0\le G\le 1$, $G = 1$ on $\left [a, b\right ]$, and $||G^{(k)}|| _\infty\ll L^{k} (\log L)^{6\alpha}$ for all $0\le k \le 6$ where $\alpha = c/7$ with $c$ being the constant in Lemma \ref{lm:reduction:01}.
  	
  	By the definition of $G$, we have  
  	\begin{equation}
  	\E N_{{f_n}}{(a, b)}\le \E \sum_{\zeta:\ f_n(\zeta)=0} G(\zeta ) \le \E N_{{f_n}}{(a-L^{-1}(\log L)^{-\alpha}, b+L^{-1}(\log L)^{-\alpha})}\nonumber.
  	\end{equation}
   	Similarly, 
  	\begin{equation}
  	\E N_{{\tilde f_n}}{(a, b)}\le \E \sum_{\tilde \zeta:\ \tilde f_n(\tilde \zeta)=0} G(\tilde \zeta) \le \E N_{{\tilde f_n}}{(a-L^{-1}(\log L)^{-\alpha}, b+L^{-1}(\log L)^{-\alpha})}\nonumber.
  	\end{equation}
  	Applying Lemma \ref{lm:reduction:01} to the function $G/(\log L)^{6\alpha}$, we get
  	\begin{eqnarray}
  	\E \sum_{\zeta:\ f_n(\zeta)=0} G(\zeta)&=&  \E \sum_{\tilde \zeta:\ \tilde f_n(\tilde \zeta)=0} G(\tilde \zeta)+  O\left ((\log L)^{-c+6\alpha}\right )= \E \sum _{\tilde \zeta:\ \tilde f_n(\tilde \zeta)=0} G(\tilde \zeta)+  O\left ((\log L)^{-\alpha}\right ) \nonumber.
  	\end{eqnarray}
  	
  	Thus,	\begin{eqnarray}
  	\E N_{{f_n}}{(a, b)}&\le& \E N_{{\tilde f_n}}{(a-L^{-1}(\log L)^{-\alpha}, b+L^{-1}(\log L)^{-\alpha})}+ O((\log L)^{-\alpha} ) \nonumber\\
  	&\le& \E N_{\tilde f_n}{(a, b)}+ 2\mathcal N_{\tilde f_n} + O((\log L)^{-\alpha} ) \nonumber,
  	\end{eqnarray}
  	where 
  	$\mathcal N_{\tilde f_n} := \sup_{x=1-\Theta(1/L)}\E N_{\tilde f_n} (x - L^{-1}(\log L)^{-\alpha}, x)$. We will show later that 
  	\begin{equation}\label{eq:CN}
  	\mathcal N_{\tilde f_n}\ll (\log L)^{-\alpha}
  	\end{equation}
  	which gives the upper bound $\E N_{{f_n}}{(a, b)}\le \E N_{\tilde f_n}{(a, b)}+ O ((\log L)^{-\alpha})$.
  	
  	By the same argument with $G$ being supported on the inner interval 
  	$$\left [a+L^{-1}(\log L)^{-\alpha}, b-L^{-1}(\log L)^{-\alpha}\right],$$ 
  	we obtain the lower bound $\E N_{{f_n}}{(a, b)}\ge \E N_{\tilde f_n}{(a, b)}-O((\log L)^{-\alpha}$ which completes the proof of \eqref{eq:uni:ex}.
  	
  	To prove \eqref{eq:CN}, we use  the Kac-Rice formula (here we use \cite[Formula 3.12]{Farahmand}), which asserts that for  every $x$, 
  	\begin{eqnarray} 
  	\E N_{\tilde f_n}[x - L^{-1}(\log L)^{-\alpha},x] &=&  \frac{1}{\pi} \int_{x - L^{-1}(\log L)^{-\alpha}}^{x} \frac{\sqrt{PQ-R^2}}{P} dx, \nonumber
  	\end{eqnarray}
  	where
  	\begin{equation}\label{key}
  	P = \sum_{i=0}^{n} c_{i, n}^2 x^{2i},\quad Q = \sum_{i=1}^{n} i^2 c_{i, n}^2 x^{2i-2},\quad \text{and}\quad  R = \sum_{i=1}^{n} i c_{i, n}^2 x^{2i-1}. \nonumber
  	\end{equation}
  	For $x = 1-\Theta(1/L)$, we have $P\ge c_{i, n}^2 x^{2i}\gg 1$ for any $i=O(1)$ with $c_{i, n}\gg 1$. Also, 
  	$$Q\ll \sum_{i=1}^{n} i^2 i^{-1} x^{2i-2} = \sum_{i=1}^{\infty} i x^{2i-2} \ll L^{2}.$$
  	 So, 
  	\begin{eqnarray} 
  	\E N_{\tilde f_n}[x - L^{-1}(\log L)^{-\alpha},x] &\ll& \int \sqrt\frac{{Q}}{P} dx\ll \int \sqrt Q dx \ll \int L  \ll (\log L)^{-\alpha}\nonumber
  	\end{eqnarray}
  	as claimed.
  \end{proof}

\subsection{Roots inside $(1, \infty)$} 
The goal of this section is to prove Lemma \ref{lm:allrho:1} which is to show that
\begin{equation}
\E N_n(1,\infty) = \frac{1+o(1)}{2\pi}  \log n, \quad \text{for all $\rho$}\nonumber.
\end{equation}
We first use the transformation
\be
g_n(x) = \frac{x^n}{c_n} f_n(1/x)\nonumber
\end{equation}
that converts the zeros of $f_n$ in $(1, \infty)$ to the zeros of $g_n$ in $(0,1)$.
Note that 
\be
g_n(x) = \sum_{m=0}^n d_m \xi'_m x^m,\nonumber
\end{equation}
where $\xi'_m:=\xi_{n-m}$ and 
$$d_m = \frac{c_{n-m}}{c_n} =  \left (\frac{n-m}{n}\right ) ^{\rho}(1+o_{n-m}(1))$$
where the $o_{n-m}(1)$ means that it goes to $0$ as $n-m\to \infty$. Observe that  
\begin{equation}\label{eq:dm}
d_m = 1+o(1) \quad \text{for $m=o(n)$}.
\end{equation}
Even though this asymptotics only holds for $m=o(n)$, it actually governs the behavior of $g_n$ in the core interval $\I_n$ because for $x\in \I_n$ and for large $m$, $x^{m}$ is tiny and does not contribute much to the sum. In other words, on $\I_n$, $g_n$ would behave like the Kac polynomial, and in particular, would have universality properties.
With this observation, the rest of the proof will be similar (and simpler) to what we have done for $(0, 1)$ in the previous section. We include it here for completeness.
Let
$$\tilde g_n = \sum_{m=0}^{n} d_m \tilde \xi_m x^{m}$$
with $\tilde \xi_m$ being iid standard Gaussian.
The goal of this subsection is to show the following analog of Lemma \ref{lm:reduction:01}.
\begin{lemma} [Reduction to Gaussian]\label{lm:reduce:gau:2} There exist positive constants $C, c'$ depending only on $\ep_0$, $C_{2+\ep_0}$ and the rate of convergence of $ |c_{m, n}|/m^{\rho}$   such that for every $x_0\in \I_n$, let 
	$r = \frac{1-x_0}{3}$ and $\delta = r$ and for every function $G: \R \to \R$ supported on $ [x_0-10^{-5}r, x_0+10^{-5}r] $ with continuous derivatives up to order $6$ and $||G^{(a)}||_{\infty}\le r^a$ for all $0\le a\le 6$, we have
	\begin{eqnarray} 
	\left |\E\sum_{g_n(\zeta) =0} G\left (\zeta  \right) 
	-\E\sum_{\tilde g_n(\tilde \zeta) =0} G\left (\tilde \zeta  \right) \right |\ll \delta^{c}\nonumber
	\end{eqnarray}
\end{lemma}
For the proof, for each $x_0$, we let $r$ and $\delta$ as above. We let  $L = r^{-1}$ and will show that there exist constants $\alpha_1, C_1$ such that Universality Conditions \eqref{cond-delocal}--\eqref{cond-poly} hold for all $x_0\in \I_n$ with the corresponding parameters  $r = \frac{1-x_0}{3}$ and $\delta = r$.
\subsubsection{Universality Conditions \eqref{cond-delocal} and \eqref{cond-repulsion}} Since $x_0\in \I_n$, for all $z\in B(x_0, r)$, we have 
$$|z|= 1-\frac{1}{\Theta(L)}, \quad |z|\le 1 - \frac{\exp((\log n) ^{1/5})}{2n}.$$ 
Let $n_0 = \frac{2n}{\exp((\log n) ^{1/5}/2)}$. We have $L = o(n_0), n_0= o(n)$.  From \eqref{eq:dm},
$$|d_i z^{i}|\ll \begin{cases}
1\quad\text{ if $i\le n_0$,}\\
n^{|\rho|}|z|^{n_0}\ll 1 \quad\text{ if $i > n_0$,}
\end{cases}$$
while 
\begin{eqnarray}
\sum _{j = 0}^{n}|d_j|^{2}|z|^{2j} \ge  \sum _{j = 0}^{L}|d_j|^{2}|z|^{2j} \gg \sum _{j = 0}^{L} 1 \gg  L. \nonumber
\end{eqnarray}
Thus, Condition \eqref{cond-delocal} holds for $\alpha_1 = 1/2$.
Similarly, one has Universality Condition \eqref{cond-repulsion} for any constant $c_0>0$.

\subsubsection{Universality Condition \eqref{cond-smallball} for any constants $A, c_0>0$}
The bound \eqref{cond-smallball} follows from a more general anti-concentration bound: there exists $\theta\in I := [-r/100,  r/100]$ such that 
\begin{equation}
\sup _{Z\in \C}\P\left (|g_n(x_0 e^{i\theta})-Z|\le \exp(-L^{c_0})\right )\ll L^{-A}.\nonumber
\end{equation}
As before, we get this by applying Lemma \ref{lmanti_concentration} with $B = 2A, \mathcal E = \{0, \dots, N-1\}, N = L$ and $\bar e = \min\{|d_j| x_0^j, j\le N-1\} \gg 1$.

Moreover, using the same argument with $N =  K\log L$, we obtain a constant $C$ such that for every $1\le K\le \sqrt {n}/\log L$,  there exists $x'= x_0 e^{i\theta}$ where $\theta\in [-r/100, r/100]$ for which
	\begin{equation}
	\P\bigg (|g_n(x')|\le \exp\big(-L^{c_0} -C K\big)\bigg )\ll \frac{1}{K^{A}L^{A}}.\label{smallball_kac:1}
	\end{equation}

\subsubsection{ Universality Condition \eqref{cond-bddn} for any constants $A, c_0>0$} We show the stronger statement that there exists a constant $C$ such that for all $x_0\in \I_n$ and all $\alpha\in (0, 1)$ which may depend on $n$, it holds that
	\begin{equation}\label{eq:cond:bddn:1}
	\P\left (\sup_{z\in B (x_0, r/8)} |g_n(z)|\le C\alpha^{-1/2} L^{3/2}\right )\ge 1 - C\alpha.
	\end{equation}
  Indeed, consider the event that $|\xi'_k|\le \sqrt k \alpha^{-1/2}$ for all $k\le n$ which, by Markov's inequality using the $(2+\ep_0)$-moments of $\xi'_k$, holds with probability at least $1 - O(\alpha)$.
	
	Under this event, we have 
	\begin{eqnarray*}
		\sup_{w\in B(x_0, 1/(2L))}|g_n (w)| &\le& \sum_{k=0}^{n}  |d_k| |\xi'_k| \left (1-  \frac{1}{\Theta(L)}\right )^{k}\\
		&\ll& \alpha^{-1/2}\sum_{k=0}^{n}  \sqrt k |d_k|  \left (1-  \frac{1}{\Theta(L)}\right )^{k}\ll   \alpha^{-1/2} L^{3/2}.
	\end{eqnarray*}

\subsubsection{Universality Condition \eqref{cond-poly}} This condition holds with $C_1 = 1$ by the exact same argument as before, using Jensen's bound, \eqref{smallball_kac:1} and \eqref{eq:cond:bddn:1} together with the fact that $N_{g_n}(B(x_0, r))\le n$ a.s.

\subsubsection{Proof of Lemma \ref{lm:allrho:1}}\label{sec:proof:-1/2:01}
This proof is similar to the proof of Lemma \ref{lm:-1/2:01}. Having proved Lemma \ref{lm:reduce:gau:2} that reduces to the Gaussian case for smooth test functions $G$, we pass to the case that $G$ is an indicator function: there exists a positive constant $\alpha$ such that for every $x_0\in \I_n$, 
\begin{equation}\label{eq:uni:ex:1}
\E N_{g_n}\left (x_0-\frac{1-x_0}{10^6}, x_0+\frac{1-x_0}{10^6}\right ) - \E  N_{\tilde g_n}\left (x_0-\frac{1-x_0}{10^6}, x_0+\frac{1-x_0}{10^6}\right ) \ll (1-x_0)^{\alpha}
\end{equation}
Assuming this and decomposing $\I_n$ into such intervals, we obtain
\begin{equation}\label{eq:uni:In:f:1}
\E N_{g_n}(\I_n)- \E N_{\tilde g_n}(\I_n) \ll \sum_{k=\Theta(\log n)^{1/5}}^{C\log n} (3/5)^{\alpha k} =o(1).
\end{equation}

In Lemma \ref{lm:main interval}, we showed that the number of roots in $(0, 1)\setminus \I_n$ of $g_n$ (and also $\tilde g_n$ as it is just a special case of $g_n$) is $o(\log n)$. Combining this with \eqref{eq:uni:In:f:1} yields
\begin{equation} 
\E N_{g_n}(0, 1)- \E N_{\tilde g_n}(0, 1)  = o(\log n).\nonumber
\end{equation}
In Section \ref{sec:gau:1inf}, we shall show that 
\begin{equation}\label{eq:gn:N}
\E N_{\tilde g_n}(0, 1)  =  \frac{1+o(1)}{2\pi}\log n.
\end{equation}
So, 
\begin{equation} 
\E N_{  f_n}(1, \infty)=\E N_{  g_n}(0, 1)  =  \frac{1+o(1)}{2\pi}\log n.\notag
\end{equation}
which completes the proof of Lemma \ref{lm:allrho:1}.

\begin{proof}[Proof of \eqref{eq:uni:ex:1}]
	Let $L = 3(1-x_0)^{-1}$. The same proof as for \eqref{eq:uni:ex} applies for $G$ being a smooth function supported in   $\left [x_0-\frac{1-x_0}{10^6}\mp L^{-1-\alpha}, x_0+\frac{1-x_0}{10^6}\pm L^{-1-\alpha}\right]$. It is left to show that
 	$$\mathcal N_{\tilde g_n} := \sup_{x=1-\Theta(1/L)}\E N_{\tilde g_n} (x - L^{-1-\alpha}, x)\ll L^{-\alpha}.$$ 
  	
 By the Kac-Rice formula, 
	\begin{eqnarray} 
	\E N_{\tilde g_n}[x - L^{-1-\alpha},x] &=&  \frac{1}{\pi} \int_{x - L^{-1 -\alpha}}^{x} \frac{\sqrt{PQ-R^2}}{P} dx\ll \int_{x - L^{-1 -\alpha}}^{x} \sqrt{\frac{Q}{P}} dx, \nonumber
	\end{eqnarray}
	where
	\begin{equation}\label{key}
	P = \sum_{i=0}^{n} d_i^2 x^{2i}\ge \sum_{i=0}^{L} d_i^2 x^{2i}\gg L,\quad Q = \sum_{i=1}^{n} i^2 d_i^2 x^{2i-2}\ll L^{3}. \nonumber
	\end{equation}
	So, 
$\E N_{\tilde g_n}[x - L^{-1}L^{-\alpha},x] \ll \int L \ll L^{-\alpha}
$ as desired.
\end{proof}

\section{Proof of Theorem \ref{thm:main:critical} for the Gaussian case} \label{sec:gau}
\subsection{The number of roots in $[0, 1]$ for Gaussian case} \label{sec:01:Gau}
In this section, we prove \eqref{eq:Gauss:01}. In particular, we prove the following.  
\begin{lemma}  \label{lm:Gaussian:01}  
	Let
	$$f_n(x) = \sum_{m=0}^n c_{m, n} \xi_m x^m, $$
	be a random polynomial where $\xi_m$ are iid standard Gaussian, and $c_{m, n}$ is a deterministic sequence satisfying
	$c_{m, n} =  m^{-1/2}(1+o(1))$, as $m \rightarrow \infty$.
	Then 
	\begin{equation}\label{eq:01:Gauss}
	\E N_n(0, 1) = \frac{1+o(1)}{\pi} \sqrt{\log n}.
	\end{equation}
\end{lemma}

 By Lemma \ref{lm:main interval}, if suffices to restrict to the interval $ \I_n$.
For this, we again use the Kac-Rice formula
\begin{eqnarray}
	\E N(\I_n)	&=& \frac{1}{\pi} \int_{\I_n} \frac{\sqrt{P(x)Q(x)-R(x)^2}}{P} dx,\label{eq:KR}
\end{eqnarray}
where
\begin{equation}\label{key}
P(x) = \sum_{i=0}^{n} c_{i, n}^2 x^{2i},\quad Q(x) = \sum_{i=1}^{n} i^2 c_{i, n}^2 x^{2i-2},\quad \text{and}\quad  R(x) = \sum_{i=1}^{n} i c_{i, n}^2 x^{2i-1}. \nonumber
\end{equation}
Using the following lemma, we will show that the integrand in \eqref{eq:KR} converges uniformly to a function that we can integrate explicitly. 

\begin{lemma}\label{lemma:In}
Uniformly on the interval $\I_n$, we have (as $n \rightarrow \infty$)
\begin{equation}\label{eq:P}
P(x) = -\log(1-x^2) \left(1+o(1) \right),
\end{equation}
\begin{equation}\label{eq:Q}
Q(x) = \frac{1}{(1-x^2)^2} \left(1+o(1) \right),
\end{equation}
and
\begin{equation}\label{eq:R}
R(x) = \frac{x}{1-x^2} \left(1+o(1) \right).
\end{equation}
\end{lemma}

Assuming Lemma \ref{lemma:In}, we obtain 
\begin{equation}\label{eq:KR2}
	\E N(\I_n) = (1+o(1)) \frac{1}{\pi} \int_{\I_n} \frac{\sqrt{\frac{-\log(1-x^2)}{(1-x^2)^2 }- \frac{x^2}{(1-x^2)^2}}}{-\log(1-x^2)} dx
\end{equation}
where we could pull out the $1+o(1)$ outside of $PQ - R^{2}$ because from  Lemma \ref{lemma:In}, $PQ-R^{2}\gg PQ + R^{2}$. So, the integral in \eqref{eq:KR2} satisfies the following asymptotic
$$\E N(\I_n) = (1+o(1))\int_{\I_n} \frac{\sqrt{\frac{-\log(1-x^2)}{(1-x^2)^2 }- \frac{x^2}{(1-x^2)^2}}}{-\log(1-x^2)} dx = \left(1+o(1) \right) \int_{\I_n} \frac{1}{2(1-x)\sqrt{-\log(1-x)}} dx,$$
where we have used that, for all $x \in \I_n$, we have $-\log(1-x) \rightarrow \infty$ as $n \rightarrow \infty$, which implies $\log(1-x) + O(1) = \left( 1+o(1) \right) \log(1-x)$.

Recall that \begin{equation} 
\I_n = \left [1 - \exp\left (-(\log n)^{1/5}\right ), 1-\exp\left ((\log n)^{1/5}\right )/n\right ]= [1-a_n, 1-b_n].\nonumber
\end{equation}
We have
\begin{align*}
\E N(\I_n) &= (1+o(1))\int_{\I_n} \frac{1}{2(1-x)\sqrt{-\log(1-x)}} dx \\
&= (1+o(1)) \left( \sqrt{-\log(b_n)} - \sqrt{-\log(a_n)} \right) \\
&=(1+o(1)) \sqrt{\log n}
\end{align*}
as desired.

\begin{proof}[Proof of Lemma \ref{lemma:In}]
	First we prove the statement in  \eqref{eq:P}. We have
	\begin{align*}
	P &= \sum_{i=0}^n (1+o_{i}(1))  \frac{1}{i}  x^{2i}   = (1+o(1))\sum_{i=1}^n  \frac{1}{i} x^{2i} + O(\sum_{i=1}^{\log\log n} \frac{1}{i} x^{2i} ).
	\end{align*} 
	We have
	\begin{align*}
	\sum_{i=1}^n  \frac{1}{i} x^{2i} &= -\log(1-x^2) - \sum_{i=n+1}^{\infty} \frac{1}{i} x^{2i} = -\log(1-x^2) (1+o(1))
	\end{align*}
because	
\begin{align*}
  \sum_{i=n+1}^{\infty} \frac{1}{i} x^{2i}  &\leq  x^{2n}\sum_{i=1}^{\infty} \frac{1}{i} x^{2i}  =o\left (-\log(1-x^2) \right).
	\end{align*}
	For the remaining term, we observe that $-\log (1-x^{2})\gg \log n$ for $x\in \I_n$. Thus,
	$$\sum_{i=1}^{\log \log n} \frac{1}{i} x^{2i} \ll \log \log n = o\left (-\log(1-x^2) \right).$$
	 By the same lines of arguments, we obtain \eqref{eq:Q} and \eqref{eq:R}. 
\end{proof}

 \subsection{The number of roots in $(1, \infty)$ for Gaussian case}\label{sec:gau:1inf}
In this section, we prove Equation \eqref{eq:gn:N}. We recall that $g_n(x) = \sum_{m=0}^n d_m \xi'_m x^m$ where $\xi'_m$ are iid standard Gaussian and 
$$d_m =  \frac{c_{n-m}}{c_n}\quad\text{for}\quad c_{m, n} = m^{\rho}(1+o(1))$$
where $\rho\in \R$ is an arbitrary constant.
We need to show that
 \begin{equation}\label{eq:1toinfinity:3}
\E N_{g_n}(\I_n) = (1+o(1))\frac{1}{2\pi} \log n.
\end{equation}

 To this end, we again use the Kac-Rice formula and obtain \begin{eqnarray}
 \E \tilde N_{g_n}(\I_n) = \frac{1}{\pi} \int_{\I_n} \frac{\sqrt{PQ-R^2}}{P} dx, \nonumber
 \end{eqnarray}
 where
 \begin{equation}\label{key}
 P = \sum_{i=0}^{n} d_i^2 x^{2i},\quad Q = \sum_{i=1}^{n} i^2 d_i^2 x^{2i-2},\quad \text{and}\quad  R = \sum_{i=1}^{n} i d_i^2 x^{2i-1}. \nonumber
 \end{equation}
Setting $n_0 = \frac{n}{\exp{(\frac{1}{2}(\log n)^{1/5}})}$,  we have for $x\in \I_n$, 
 \begin{eqnarray}
 P(x) &=& \sum_{i=0}^{n_0} d_i^2 x^{2i} + \sum_{i=n_0+1}^{n} d_i^2 x^{2i} = (1+o(1))\sum_{i=0}^{n_0} x^{2i} + \sum_{i=n_0+1}^{n} d_i^2 x^{2i}.\nonumber
 \end{eqnarray}
 Since  $d_m = O(n^{|\rho|})$ for all $m$, we have
 \begin{eqnarray}
 \sum_{i=n_0+1}^{n} d_i^2 x^{2i} \ll n^{|\rho|}  \sum_{i=n_0+1}^{n} b_n^{2i} \ll n^{|\rho|} \exp\left (-\exp(\frac{1}{2}\left (\log n)^{1/5}\right )\right ) =o\left (\sum_{i=0}^{n_0} x^{2i}\right ).\nonumber
 \end{eqnarray}
 And so, uniformly for $x\in \I_n$,
  \begin{eqnarray}
 P(x) &=& (1+o(1))\sum_{i=0}^{n_0} x^{2i} = \frac{1+o(1)}{1-x^{2}}=\frac{1+o(1)}{2(1-x)}.\nonumber
 \end{eqnarray}
 Similarly, uniformly for $x\in \I_n$,
 $$Q =\frac{2+o(1)}{(1-x^{2})^{3}} = \frac{1+o(1)}{4(1-x)^{3}} \quad \text{and}\quad R = \frac{1+o(1)}{(1-x^{2})^{2}} =\frac{1+o(1)}{4(1-x)^{2}} $$
 So, 
 $$\E \tilde N_{g_n}(\I_n) = (1+o(1))\frac{1}{2\pi} \int _{\I_n} \frac{dx}{1 - x} =\frac{1+o(1)}{2\pi}\log n.$$

\section{Proof of Theorem \ref{thm:main:sub}}\label{sec:main:sub}
Firstly, we observe that \eqref{eq:sub:1inf} is just a special case of Lemma \ref{lm:allrho:1}. Equation \eqref{eq:sub:R} is a simple corollary of \eqref{eq:sub:01} with $\ell=1$ and \eqref{eq:sub:1inf}. Thus, it remains to show \eqref{eq:sub:01}. By the symmetry observed at the beginning of Section \ref{sec:pf:main}, we can reduce to the interval $[0, 1]$ and show that 
$$\E N^{\ell}_{f_n}[0, 1] = O_{\ell}(1).$$

To this end, we decompose the interval $[0, 1]$ into three intervals $I_1=[0, 1-\frac{1}{C}]$, $I_2=[1-\frac{1}{C}, 1-\frac1n]$ and  $I_3=[1-\frac{1}{n}, 1]$, where $C$ is a large constant (say, $C = 100$). We shall show that in each interval, the number of roots has a bounded $\ell$-moment. 

\subsection{Roots in $I_1=[0, 1-1/C]$} 
Let $k_0 = \lceil -\rho\rceil$. By  Lemma \ref{lm:interlace}, we have
$$\E N^{\ell}_{f_n}(I_1)\le \E (k_0+   N_{f^{(k_0)}_n}(I_1))^{\ell}\le 2^{\ell-1}(k_0^\ell+ \E (N_{f^{(k_0)}_n}(I_1))^\ell).$$

By the choice of $k_0$, we note that $f^{(k_0)}_n$ falls in the super-critical regime which allows us to use available results for this case. In particular, from \cite{nguyenvuCLT}, it follows (see display (65) in section~5.2 therein) that $\E[(N_{f^{(k_0)}_n}(I_1))^\ell]=O_{\ell}(1)$. Hence, $\E[(N_{f_n}(I_1))^\ell]=O_{\ell}(1)$ as desired.

\subsection{Roots in $I_2=(1-\frac{1}{C}, 1-\frac{1}{n}]$} This is the main interval where the expected number of roots is $\Theta(\log n)$ for $\rho>-\frac12$ and our goal is show that it is $O(1)$ for the case $\rho<-\frac12$.  

We shall prove in the next lemma a key estimate about the probability that there are multiple roots in a small interval. We prove this by adapting a classical approach by Erd\H{o}s-Offord \cite{EO} (see also Ibragimov-Maslova \cite{Ibragimov1971expected1}).

\begin{lemma}\label{lm:multipleroot}  
	There exist positive constants $C_{0}, a_1,\lambda$ depending only on $\rho$, $\ep_0$, $C_{2+\ep_0}$ and the rate of convergence of $|c_{m, n}| m^{\rho}$,  such that for any integer $r$, for  all $1\le k \leq n^{\lambda/r}$ and all $x, y \in (1-\frac{1}{C}, 1-\frac{1}{n}]$ satisfying
	$$
	\log\frac{1-x}{1-y} = \frac{1}{2C_0},
	$$
	we have
	\begin{equation}\label{ine:multiple1}
	\P (N_n(x, y)\geq k)\le \frac{2}{( k\log 2+ a_1\log\frac{1}{1-y})^{r}}.
	\end{equation}
	Moreover, for every integer $\ell>0$, there exists a constant $C_{\ell}$ depending only on $\ell, \rho$, $\ep_0$, $C_{2+\ep_0}$, and the rate of convergence of $|c_{m, n}| m^{\rho}$ such that
	\begin{equation}\label{eq:tail:ex}
	\E N^{\ell}_n(x, y)\le C_{\ell}|\log (1-y)|^{\ell}\le  C_{\ell}(\log n)^{\ell}.
	\end{equation}
\end{lemma}

Assuming Lemma \ref{lm:multipleroot}, let $\delta = 1/(2C_0)$.
We cover $I_2$ by intervals of the form $J_j:=(1-e^{-\delta(j-1)}, 1 - e^{-\delta j}]$ for $j=C, C+1, \dots, C+m-1$ with $m=\frac{1}{\delta}\log n$. From  Lemma \ref{lm:multipleroot}, 
for $1\le k \leq n^{\lambda/r}$,
\begin{equation*} 
	\P (N_{n}(J_j)\geq k)\le \frac{2}{(k\log 2+ja_1\delta )^{r}}.
\end{equation*}
Therefore, for any positive integers $\ell,r$ with $r>\ell+1$,
\begin{align*}
	\E[N_n(J_j)^\ell] &\le \sum_{k=0}^{n^{\lambda/r}}(k+1)^\ell \ \P (N_n(J_j)\ge k)  + \E[N_n(J_j)^{\ell}{\mathbf 1}_{N_n(J_j)\ge n^{\lambda/r}}] \\
	&\le \sum_{k=0}^{n^{\lambda/r}}\frac{2(k+1)^\ell}{(k\log 2+ja_1\delta )^{r}} +  C_\ell\frac{(\log n)^{\ell}}{n^{\lambda/r}} \text{ by \eqref{eq:tail:ex}}\\
	&\le C_\ell\sum_{k=0}^{\infty}\frac{1}{(k+j)^{r-\ell}}+C_\ell\frac{(\log n)^{\ell}}{n^{\lambda/r}} \\
	&\le C_\ell\frac{1}{j^{r-\ell-1}}+C_\ell\frac{(\log n)^{\ell}}{n^{\lambda/r}}.
\end{align*}
We note that the constant $C_{\ell}$ changes from one equation to another. To bound the higher moments of $N_n(J)\le\sum_{j=1}^mN_n(J_j)$, we use H\"{o}lder's inequality to write
$$
N_n(J)^{\ell}\le \left(\sum_{j=1}^m\frac{1}{j^{\ell/(\ell-1)}}\right)^{\ell-1}\left( \sum_{j=1}^m j^\ell N_n(J_j)^\ell\right) \; \le \; C_\ell\sum_{j=1}^m j^\ell N_n(J_j)^\ell.
$$
Therefore,
\begin{align*}
	\E[N_n(J)^\ell] &\le C_\ell \sum_{j=1}^m j^\ell\left(\frac{1}{j^{r-\ell-1}}+\frac{(\log n)^{\ell}}{n^{\lambda/r}}\right) \\
	&\le C_\ell\left(\sum_{j\ge 1}\frac{1}{j^{r-2\ell -1}} + m^{\ell+1}\frac{(\log n)^{\ell}}{n^{\lambda/r}}\right).
\end{align*}
As $m=2 C_0\log n$, choosing $r=2\ell+3$, we see that $\E[N_n(J)^\ell]=O(1)$.

\begin{proof}[Proof of Lemma \ref{lm:multipleroot}]  By Rolle's theorem and the fundamental theorem of calculus, if $f_n$ has at least $k$ zeros in the interval $(x, y)$ then
	$$
	|f_n(y)|\leq\int_{x}^{y}\int_{x}^{y_{1}}\cdots\int_{x}^{y_{k-1}}|f_n^{(k)}(y_{k})|dy_{k}\ldots dy_{1}=:I_{x,y}.
	$$

	Therefore,
	\begin{align*}
	\P (N_n(x, y)\geq k) \quad\leq\quad& \P \left(I_{x,y}\geq\varepsilon_{1}\sqrt{V(y)}\right) \; +\;  \P \left(|f_n(y)|\leq\varepsilon_{1}\sqrt{V(y)}\right)
	\end{align*}
	where $\varepsilon_{1}>0$ to be chosen and 
	\begin{equation}\label{varbound:1}
	V(y)   =\Var f_n(y)=\sum_{i=0}^{n}c_{i}^{2}y^{2i} =\Theta(1) \quad\text{because $\rho<-1/2$}.
	\end{equation}

	We shall show later for some positive constants $C_0, a_1, a_2, \gamma$,
	\begin{equation}\label{eq:int:concentration}
	\P \left(I_{x,y}\geq\varepsilon_{1}\sqrt{V(y)}\right) \le  \ep_1^{-2} 4^{-k} (1-y)^{a_1},
	\end{equation}
	and for any $r>0$, 
	\begin{equation}\label{eq:int:anticoncentration}
	\P \left(|f_n(y)|\leq\varepsilon_{1}\sqrt{V(y)}\right) \le |\log \ep_1|^{-r}
	\end{equation}
	provided $\ep_1\ge Ce^{-n^{\gamma/r}}$ and $\ep_1\le 1/C$.
	
	%
	Assuming these bounds, take $\ep_1$ so that the right-hand sides of \eqref{eq:int:concentration} and \eqref{eq:int:anticoncentration} are equal. In particular, take 
	$$
	\ep_1 = 2^{-k} (1-y)^{\frac{a_1}{2}} \ \bigg|\log\left( 4^{-k} (1-y)^{a_1} \right) \bigg|^{\frac{r}{2}}.
	$$  
	For this choice to be valid, we must have $\ep_1\ge Ce^{-n^{\gamma/r}}$, which follows from $1-y\ge \frac{1}{n}$ provided we take  $k\le n^{\lambda/r}$  where $\lambda=\gamma/2$ (say).
	
	We then have the following bound for $k\le n^{\lambda/r}$ and $y<1-\frac{1}{n}$,
	\begin{equation}\label{key}
	\P (N_n(x, y)\geq k)\le 2|\log\ep_1|^{-r} \le \frac{2}{( k\log\frac{1}{C_{0}\delta}+ a_1\log\frac{1}{1-y})^{r}}
	\end{equation}
	completing the proof of \eqref{ine:multiple1}.

	To prove \eqref{eq:tail:ex}, we apply \cite[Lemma 3.1]{nguyenvuCLT} which is a version of Lemma \ref{lm:multipleroot} for the super-critical regime. So, we again use  Lemma \ref{lm:interlace} to reduce to this regime,
	$$\E N^{\ell}_{f_n}(x, y)\le \E (k_0+   N_{f^{(k_0)}_n}(x, y))^{\ell}\le 2^{\ell-1}(k_0^\ell+ \E (N_{f^{(k_0)}_n}(x, y))^\ell)$$
	where $k_0 = \lceil -\rho\rceil$. Now, we use \cite[Lemma 3.1]{nguyenvuCLT} in which we choose $\delta$ to be $1-y$ and $z$ to be $1-3\delta/2$. With these choices and the observation that the complex ball centered at $1-3\delta/2$ and radius $(1-y)/2$ contains the interval $(x, y)$ and so, 
	$$\E (N_{f^{(k_0)}_n}(x, y))^\ell\le C_{\ell}|\log (1-y)|^{\ell},$$
	proving \eqref{eq:tail:ex}.
\end{proof}

\begin{proof} [Proof of \eqref{eq:int:concentration}]
	By Markov's inequality, we have
	\begin{eqnarray}
	&&\left(\varepsilon_{1}\sqrt{V(y)}\right)^{2}\P \left(I_{x,y}\geq\varepsilon_{1}\sqrt{V(y)}\right) \le  \E \left(\int_{x}^{y}\int_{x}^{y_{1}}\cdots\int_{x}^{y_{k-1}}|f_n^{(k)}(y_{k})|dy_{k}\ldots dy_{1}\right)^{2}.\nonumber
	\end{eqnarray}
	By H\"older's inequality, the right-hand side is at most
	$$ \frac{(y-x)^{k}}{k!} \E \int_{x}^{y}\int_{x}^{y_{1}}\cdots\int_{x}^{y_{k-1}}|f_n^{(k)}(y_{k})|^{2}dy_{k}\ldots dy_{1}$$
	and so
	\begin{eqnarray}
	\left(\varepsilon_{1}\sqrt{V(y)}\right)^{2}\P \left(I_{x,y}\geq\varepsilon_{1}\sqrt{V(y)}\right) \le \left(\frac{(y-x)^{k}}{k!}\right)^{2}\sup_{w\in(x,y)}\E |f_n^{(k)}(w)|^{2}. \label{intbd2}
	\end{eqnarray}
	
	We have for some constant $C_0$, 
	\begin{eqnarray*}
		\sup_{w\in(x,y)} \E |f_n^{(k)}(w)|^{2} &\le& C_0\sum_{m=0}^{\infty} m^{2\rho} m^{2}\dots (m-k+1)^{2} y^{2m}\\
		&\le& C_0\sum_{m=0}^{\infty} m^{2\rho'} m^{2}\dots (m-k+1)^{2} y^{2m}\le C_0\sum_{m=0}^{\infty} m^{2\rho'+2k} y^{2m}\\
		&\le&  \frac{C_0^{k}(k!)^{2}}{(1-y)^{2\rho'+2k+1}} 
	\end{eqnarray*}
	where $\rho'$ is any constant in $(-1, -1/2)$ such that $\rho'\ge\rho$. We note that $2\rho'+2k+1\ge 0$ for all $k\ge 1$ as $\rho'>-1$.
	So,
	\begin{eqnarray*}
		\P \left(I_{x,y}\geq\varepsilon_{1}\sqrt{V(y)}\right)&\le&  \ep_{1}^{-2}\cdot \frac{(y-x)^{2k}}{(k!)^{2}} \cdot\frac{C_0^{k}(k!)^{2}}{(1-y)^{2\rho'+2k+1}} \le  \ep_{1}^{-2}\left(C_0\frac{y-x}{1-y}\right)^{2k} (1-y)^{-(2\rho'+1)}\\
		&=&\ep_{1}^{-2}\left(C_0(e^{\delta}-1)\right)^{2k} (1-y)^{-(2\rho'+1)} \le \ep_{1}^{-2}\left(C_0\delta\right)^{2k} (1-y)^{-(2\rho'+1)}
	\end{eqnarray*}
	where the constant $C_0$ may change from one equation to another.
	This proves \eqref{eq:int:concentration} with $a_1 = -(2\rho'+1)>0$.
\end{proof}

\begin{proof}[Proof of \eqref{eq:int:anticoncentration}] We shall use the following result (Lemma~4.2 in \cite{DOV}). 
\begin{lemma} Assume that $\xi_i$ are independent random variables satisfying Assumption-A with parameters $\ep_0, C_{2+\ep_0}$. Then there are constants $C',\alpha$ depending only on $\ep_0, C_{2+\ep_0}$ such that for any complex numbers $b_0,\ldots ,b_n$ containing a lacunary subsequence $|b_{\ell_1}|\ge 2|b_{\ell_2}|\ge \ldots \ge 2^m|b_{\ell_m}|$, and any $z\in \C$, we have 
	$$
	\P\left\{\big |\sum_{i=0}^n b_i\xi_i-z \big|\le |b_{\ell _m}|  \right\}\le C'e^{-\alpha m}.
	$$ 
\end{lemma}
To apply this to our situation, fix $y\in [1-\frac{1}{C_0},1)$. Let $\ell_1=1$ and for $i\ge 1$ define
$$
\ell_{i+1}=\min\{j>\ell_{i}\; : \; |c_{j, n}|y^j\le \frac12 |c_{\ell_i, n}|y^{\ell_i}\}.
$$
For the rest of this proof, we drop the second subscript $n$ for brevity.

Let $b_{i}= c_{i}y^i$. Then $(b_{\ell_i})$ is a lacunary sequence. 
Assume that $C_0\ge 3$ without loss of generality. Since $|c_{k}/c_{k+1}|=1+o(1)$ as $k\to\infty$, we can choose $i_0$ such that $|c_{i}|\le \frac43 |c_{i+1}|$ for $i\ge i_0$. Then (as $\ell_i\ge i$), for $i\ge i_0$,
\begin{equation}\label{eq:b}
|b_{\ell_i}|/2\ge |b_{\ell_{i+1}}|= \left |\frac{c_{\ell_{i+1}}y}{c_{\ell_{i+1}-1}}\cdot c_{\ell_{i+1}-1}y^{\ell_{i+1}-1}\right |\ge \frac43 y\cdot \frac12|b_{\ell_i}|\ge \frac14 |b_{\ell_i}|.
\end{equation}
 The maximum length of $(b_{\ell_i})_{i\ge1}$ that we can use is  $M =\max\{m \ : \ \ell_m\le n\mbox{ and } |b_m|\ge \ep_1\|c\|_2\}$ where $\|c\|^{2}_2=\sum_{i= 0}^{n}c_{i, n}^2\gg 1$ is an upper bound for $V(y)$. By conditioning on all $\xi_j$ for $j\not\in \{\ell_0,\ell_1,\ldots \}$, and applying the above lemma, we see that
$$\P \left(|f_n(y)|\leq\varepsilon_1 \sqrt{V(y)}\right) \le C'e^{-\alpha M}.$$
Note that by \eqref{eq:b}, $|b_M|\gg 4^{-M}|b_1|\gg 4^{-M}$. So, for some positive constant $\gamma'$,
$$e^{-\alpha M}\ll 4^{-\gamma' M}\ll |b_M|^{\gamma'}\ll ((\ep_1\|c\|_2)\vee |c_n|y^n)^{\gamma'}.$$
Thus,
\begin{align*}
\P \left(|f_n(y)|\leq\varepsilon_1 \sqrt{V(y)}\right) &\le   ((\ep_1\|c\|_2)\vee |c_n|y^n)^{\gamma'}  \le C|\log \ep_1|^{-r}
\end{align*}
provided $|c_n|y^n\le |\log \ep_1|^{-r/\gamma'}$. As $|c_n|\sim n^{\rho}$ and $y<1$, this is true for $\ep_1\ge Ce^{-n^{\gamma/r}}$ where $\gamma=- \rho \gamma'$. This completes the proof of \eqref{eq:int:anticoncentration}.
\end{proof}

 \subsection{Roots in $I_3=[1-\frac{1}{n},1]$}  The argument for this part is almost identical to the previous interval $I_2$. We shall obtain an analog of Lemma \ref{lm:multipleroot}.
 \begin{lemma}\label{lm:multipleroot:n}  
 	There exist positive constants $C_{0}, a_1,\lambda$ depending only on $\rho$, $\ep_0$, the rate of convergence of $|c_{m, n}| m^{\rho}$,  such that for any integer $r$ and  all $1\le k \leq n^{\lambda/r}$,
 	we have
 	\begin{equation}\label{ine:multiple1:n}
 	\P (N_n(I_3)\geq k)\le \frac{2}{( k+\log n)^{r}}.
 	\end{equation}
 	Moreover, for every integer $\ell>0$, there exists a constant $C_{\ell}$ depending only on $\ell, \rho$, $\ep_0$, $C_{2+\ep_0}$, and the rate of convergence of $|c_{m, n}| m^{\rho}$ such that
 	\begin{equation}\label{eq:tail:ex:n}
 	\E N^{\ell}_n(I_3)\le C_{\ell}(\log n)^{\ell}.
 	\end{equation}
 \end{lemma} 

Assuming this lemma, we can use the same argument as what follows Lemma \ref{lm:multipleroot} to conclude that 
 \begin{equation} 
 \E N^{\ell}_n(I_3)=O(1)\nonumber
 \end{equation}
 which concludes the proof of Theorem \ref{thm:main:sub}.
 
\begin{proof}[Proof of Lemma \ref{lm:multipleroot:n}]
	 We run the same argument as in the proof of Lemma \ref{lm:multipleroot} for the interval $[x, y] := [1-1/n, 1]$. In particular, we obtain the following analog of \eqref{eq:int:concentration} and \eqref{eq:int:anticoncentration},
	 \begin{equation}\label{eq:int:concentration:n}
	 \P \left(I_{x,y}\geq\varepsilon_{1}\sqrt{V(y)}\right) \le  \ep_1^{-2} \frac{1}{n^{-(2\rho'+1)} (k!)^2},
	 \end{equation}
	 and for any $r>0$, 
	 \begin{equation}\label{eq:int:anticoncentration:n}
	 \P \left(|f_n(y)|\leq\varepsilon_{1}\sqrt{V(y)}\right) \le |\log \ep_1|^{-r}
	 \end{equation}
	 provided $\ep_1\ge Ce^{-n^{\gamma/r}}$ and $\ep_1\le 1/C$.
	 
	 From here, the rest of the proof is identical to that of Lemma \ref{lm:multipleroot}.
\end{proof}

 \textbf{Acknowledgments.} We thank Zakhar Kabluchko for sharing with us the result obtained in Flasche's thesis.

\bibliographystyle{abbrv}
\bibliography{H16}

\end{document}